\documentclass[11pt]{article}

\usepackage[T1]{fontenc}
\usepackage{lmodern}
\usepackage[margin=1in]{geometry}
\usepackage{amsmath,amssymb,amsthm}
\usepackage{mathtools}
\mathtoolsset{showonlyrefs}
\usepackage{graphicx}
\usepackage{booktabs,tabularx,array}
\usepackage{enumitem}
\usepackage{float}
\usepackage{microtype}
\usepackage[font=small,labelfont=bf]{caption}
\usepackage{xcolor}
\usepackage[authoryear,round]{natbib}

\definecolor{linkcol}{HTML}{0A4D8C}
\definecolor{citecol}{HTML}{0F6B5F}
\definecolor{urlcol}{HTML}{0E5F9E}
\usepackage{hyperref}
\hypersetup{
  colorlinks=true,
  linkcolor=linkcol,
  citecolor=citecol,
  urlcolor=urlcol,
  breaklinks=true,
  pdfborder={0 0 0},
  pdftitle={Wasserstein Policy Gradient for Entropy-Regularized Linear-Quadratic Control},
  pdfauthor={Zhaoyu Zhu, Rui Gao, and Shuang Li},
  pdfsubject={Exact closure on linear-Gaussian policies and global convergence of Wasserstein policy gradient for entropy-regularized linear-quadratic control},
  pdfkeywords={Wasserstein policy gradient, linear-quadratic control, entropy regularization, optimal transport, global convergence}
}

\setlist[itemize]{leftmargin=1.45em,itemsep=0.25em,topsep=0.4em}
\setlist[enumerate]{leftmargin=1.75em,itemsep=0.2em,topsep=0.4em}

\newtheorem{theorem}{Theorem}[section]
\newtheorem{proposition}{Proposition}[section]
\newtheorem{lemma}{Lemma}[section]
\newtheorem{corollary}{Corollary}[section]
\theoremstyle{definition}
\newtheorem{definition}{Definition}[section]

\newtheorem{assumption}{Assumption}[section]
\theoremstyle{remark}
\newtheorem{remark}{Remark}[section]

\newcommand{\KL}{\mathrm{KL}}
\newcommand{\Tr}{\operatorname{Tr}}
\newcommand{\cA}{\mathcal{A}}
\newcommand{\cT}{\mathcal{T}}
\newcommand{\cD}{\mathcal{D}}
\newcommand{\cG}{\mathcal{G}}
\newcommand{\dd}{\mathrm d}
\newcommand{\E}{\mathbb E}
\newcommand{\R}{\mathbb{R}}
\newcommand{\cN}{\mathcal{N}}
\newcommand{\bbS}{\mathbb{S}}

\newcommand{\norm}[1]{\left\lVert #1 \right\rVert}
\newcommand{\Vtwo}{\mathcal{V}_2}
\newcommand{\Cov}{\operatorname{Cov}}
\newcommand{\Pop}{\mathcal{P}}
\newcommand{\Scorr}{\mathsf{S}}
\newcommand{\proofpart}[1]{\par\medskip\noindent\textbf{\emph{#1.}}\ }

\begin{document}

\title{Wasserstein Policy Gradient for Entropy-Regularized Linear-Quadratic Control}

\author{
Zhaoyu Zhu\textsuperscript{1}\quad
Rui Gao\textsuperscript{2}\quad
Shuang Li\textsuperscript{3}\\[0.55em]
\small
\textsuperscript{1}University of Illinois Urbana-Champaign \quad
\textsuperscript{2}The University of Texas at Austin\\
\small
\textsuperscript{3}The Chinese University of Hong Kong, Shenzhen\\[0.35em]
\small
\texttt{zhaoyuz4@illinois.edu}\quad
\texttt{rui.gao@mccombs.utexas.edu}\quad
\texttt{lishuang@cuhk.edu.cn}
}
\date{}

\maketitle

\begin{abstract}
Wasserstein policy gradient (WPG) updates state-conditional action laws by transport in the action space. We study entropy-regularized discounted linear-quadratic (LQ) control. A Bellman verification argument shows that the unrestricted problem has a linear-Gaussian optimal policy, and the discounted-occupancy-weighted statewise Wasserstein gradient is tangent to this policy class. WPG therefore reduces exactly to a finite-dimensional ODE for the feedback gain and action covariance. We prove that this ODE is globally well posed and converges exponentially from every admissible initialization. For each fixed LQ problem, the exponent has a positive limit as the entropy temperature tends to zero and contains no perturbative factor of the form \(\exp(-c/\tau)\), while retaining the usual dependence on the conditioning of the control problem. 
% The covariance estimate is naturally weighted by the entropy temperature.
\end{abstract}

% \noindent\textbf{Keywords:}
% Wasserstein policy gradient; linear-quadratic control; entropy regularization;
% policy optimization; optimal transport; global convergence.
% \medskip

\section{Introduction}\label{sec:intro}

The geometry used to update a policy is central to continuous-control reinforcement learning. Classical policy gradient optimizes a finite-dimensional actor in Euclidean parameter space. Natural policy gradient and mirror descent instead use information geometry, usually through the KL divergence. Wasserstein policy gradient (WPG) takes a different approach: for each state, it transports the conditional action distribution in the action space. With entropy regularization, the resulting Wasserstein policy gradient flow is a Fokker-Planck equation whose drift is generated by the action gradient of the current soft action-value function and whose diffusion is generated by entropy. This formulation connects policy optimization with Wasserstein gradient flows, particle methods, and optimal-transport trust-region and proximal methods
\citep{zhang2018wgf,arbel2020kwng,moskovitz2020efficientwng,
terpin2022ottrpo,song2024metrictr,pfau2025wpo,zhu2026wasserstein}.

General WPG convergence theory is largely based on Langevin-dynamics analysis of the policy distribution and its moving Gibbs target \citep{zhu2026globalwpg,siska2026convergence}. These results rely on boundedness assumptions on rewards that exclude unbounded quadratic costs and, importantly, do not establish well-posedness of the underlying continuous-time flow. 
The convergence rate is expressed in terms of a uniform log-Sobolev coefficient for the moving Gibbs law and deteriorates exponentially in the rate of $O(\frac{1}{\tau})\exp(O(-\frac{1}{\tau}))$ as the entropy temperature \(\tau\) tends to zero. 

In this paper, we show that these general results can be sharpened in the linear-quadratic (LQ) setting. 
We proceed in two steps. First, we solve the unrestricted entropy-regularized control problem over admissible stationary randomized Markov policies. A Bellman verification argument, based on the soft Bellman equation and the discounted Riccati equation, identifies a linear-Gaussian policy
$
    \pi_\star(\cdot\mid x)=\cN(-K_\star x,\Sigma_\star)
$
is optimal over the full admissible policy class. Therefore, without loss of optimality, we formulate the learning problem over
\[
    \pi_{K,\Sigma}(\cdot\mid x)=\cN(-Kx,\Sigma).
\]
Note that this is an optimality statement rather than a distributional assumption on the state process: the initial state and disturbances need not be Gaussian.

Second, we make the policy-space geometry explicit. Tangent perturbations are represented by action-space continuity equations, and their squared norm is defined as the corresponding statewise 2-Wasserstein kinetic energy averaged with respect to the normalized discounted state-occupancy measure. Combining this metric with the policy-gradient theorem yields the Wasserstein policy-gradient flow \citep{pfau2025wpo,zhu2026globalwpg,zhu2026wasserstein,siska2026convergence}
\[
\partial_t\pi_t(u\mid x)
=
\nabla_u\cdot\left(
\pi_t(u\mid x)
\nabla_u\bigl(Q^{\pi_t}(x,u)+\tau\log\pi_t(u\mid x)\bigr)
\right).
\]
For every admissible linear-Gaussian policy, the soft action-value function is quadratic in the action. Its action gradient is therefore affine, and the WPG vector field is tangent to the linear-Gaussian policy class. Under the parameterization \((K,\Sigma)\), the flow closes exactly as
\begin{equation}\label{eq:closed_parameter_flow}
    \dot K=-2E_K,
    \qquad
    \dot\Sigma=-2M_K\Sigma-2\Sigma M_K+2\tau I.
\end{equation}
Thus, the gain-covariance ODE is an exact representation of the policy-space gradient flow on the linear-Gaussian class, rather than a projection or approximation. Every solution of this closed parameter system generates a linear-Gaussian solution of WPGF. This exact closure allows the convergence analysis to proceed through the global well-posedness of the parameter ODE, without requiring a separate uniqueness theory for all solutions of the unrestricted nonlinear Fokker-Planck equation.

The closed parameter equations also make the contrast with Fisher-Rao flow transparent. Writing \(E_K\) for the feedback Bellman residual, WPG satisfies \(\dot K=-2E_K\), while Fisher-Rao flow satisfies \(\dot K=-2\Sigma E_K\). Fisher-Rao therefore scales learning in each action direction by the current exploration variance in that direction. If a fixed total variance is spread over many redundant action coordinates, the variance in task-relevant directions decreases and Fisher-Rao slows down, whereas the WPG gain equation is unchanged. Figure~\ref{fig:ambient_dimension} illustrates this difference.

\begin{figure}[!h]
    \centering
    \includegraphics[width=0.7\textwidth]{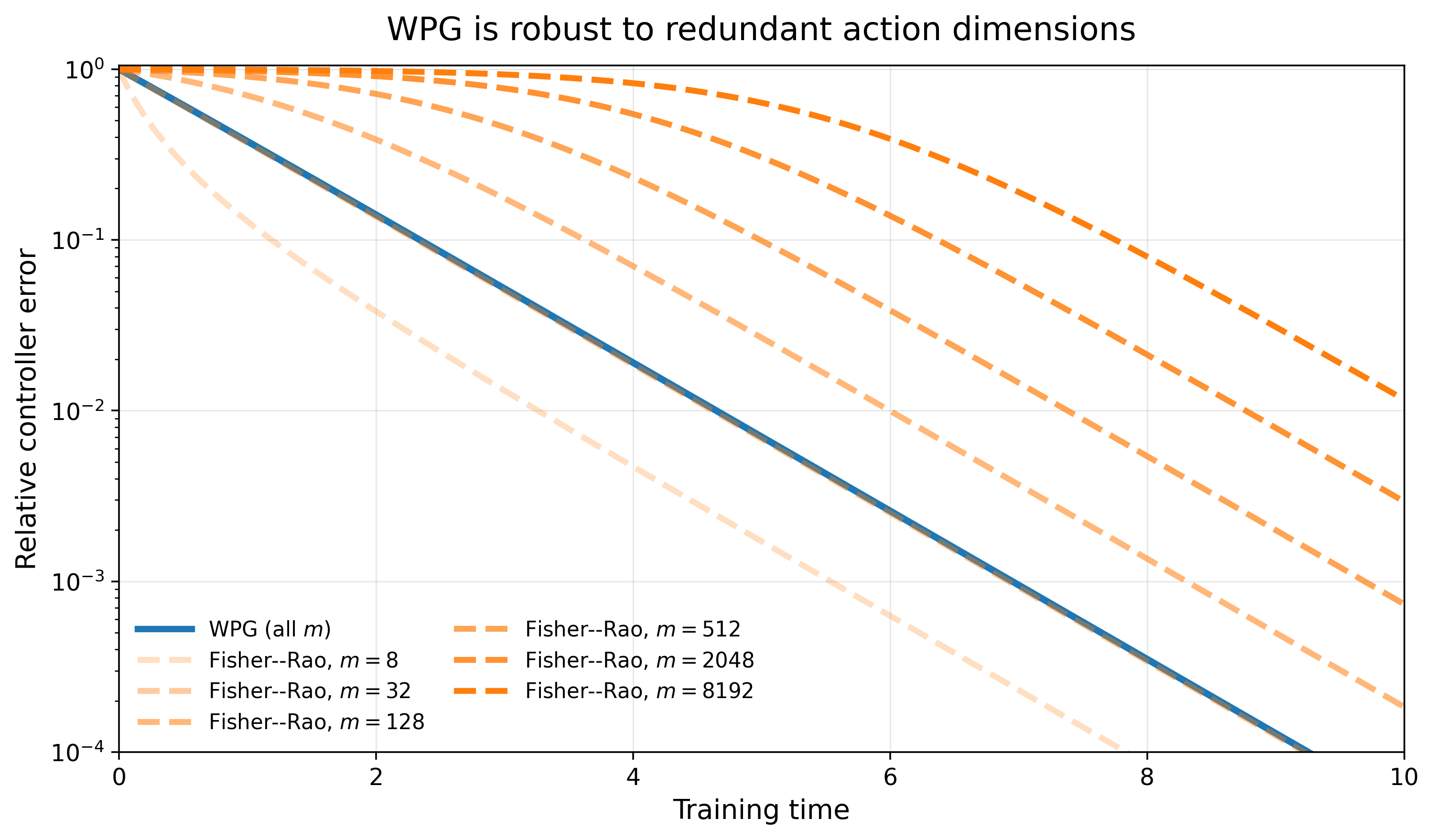}
    \caption{\textbf{Robustness to redundant action dimensions.}
    The plot gives a closed-form local comparison near the optimal policy. An eight-dimensional task-relevant action subspace is embedded in ambient dimensions $m\in\{8,32,128,512,2048,8192\}$. The optimal covariance eigenvalues in this subspace are $10^{-(i-1)/7}$, $i=1,\ldots,8$, and each redundant coordinate has optimal variance $10^{-6}$. Both flows start from the same isotropic covariance with total variance equal to ten times the optimal total variance, and the horizontal axis uses the rescaled time $s=\tau t$. As $m$ grows, the fixed variance budget is spread over more redundant coordinates. The WPG gain equation is independent of the current covariance, so all WPG curves coincide. The Fisher-Rao gain equation is scaled by the covariance, so learning in the relevant directions becomes slower.}
    \label{fig:ambient_dimension}
\end{figure}

\paragraph{Contributions.}
Our contributions are threefold.
\begin{itemize}

\item \emph{Linear-Gaussian optimality and exact policy-space reduction.}
We first consider the unrestricted class of admissible stationary randomized Markov policies and prove, through an exact relative-entropy identity, that an optimal policy is linear-Gaussian. Restricting the optimization problem to admissible linear-Gaussian policies therefore entails no loss of optimality. We then define the discounted-occupancy-weighted statewise \(2\)-Wasserstein metric on conditional action distributions and derive WPGF as the negative gradient flow of the objective under this metric. At every admissible linear-Gaussian policy, the soft action-value function is quadratic in the action, so the WPG vector field is tangent to the linear-Gaussian class and admits the exact parameter representation \eqref{eq:closed_parameter_flow}.
% Thus, the gain-covariance ODE is an exact restriction of the policy-space gradient flow, rather than a projection or approximation.
Using the Lyapunov representation of policy evaluation and compact objective sublevel sets, we further prove that this parameter system has a unique global admissible solution from every finite-cost linear-Gaussian initialization.

\item \emph{Global exponential convergence with explicit temperature dependence.}
We prove that the objective gap satisfies
\[
C(K_t,\Sigma_t)-C(K_\star,\Sigma_\star)
\le
\exp\left(
-\frac{4\lambda_{\min}(\Gamma_0)\lambda_{\min}(R)}{\norm{\Scorr_\star(\tau)}}t
\right)
\bigl(
C(K_0,\Sigma_0)-C(K_\star,\Sigma_\star)
\bigr).
\]
Here \(\Gamma_0\) is the initial-state second-moment matrix, $R$ is the action-cost matrix, and \(\Scorr_\star(\tau)\) is the discounted state-correlation matrix under the optimal policy. 
For each fixed LQ problem,
\[
\Scorr_\star(\tau)
=
\Scorr_\star^{(0)}
+
\tau\Scorr_\star^{(1)},
\]
and hence the convergence exponent approaches a positive limit as
\(\tau\downarrow0\). The LQ-specific analysis therefore avoids the additional perturbative factor \(\exp(-c/\tau)\) that appears in general WPG bounds. The exponent retains the standard dependence on the discounted state correlation, initial-state excitation, action-cost curvature, and problem scaling. The same objective estimate also controls the squared feedback-gain error and the entropy-weighted relative covariance error.

\item \emph{An exact LQ analysis without perturbative or density-ratio losses.}
The proof exploits three identities specific to the LQ setting. The Gibbs law induced by the current soft action-value function is exactly Gaussian. 
% Under the convention
% \[
% \mathcal I(\nu\Vert p)\ge 2\alpha\,\KL(\nu\Vert p),
% \]
Its sharp log-Sobolev coefficient is
\[
\alpha_K
=
\frac{2\lambda_{\min}(M_K)}{\tau}.
\]
% while the convergence proof uses the uniform lower bound
% \[
% \underline{\alpha}
% =
% \frac{2\lambda_R}{\tau},
% \qquad
% M_K\succeq R.
% \]
Combining this bound with the \(\tau^2\)-weighted Fisher-information dissipation and the identity between the Bellman residual and \(\tau\KL\) cancels the explicit temperature factor without invoking a Holley-Stroock perturbation argument. 
% Second, the Bellman resolvent identity converts the statewise Wasserstein dissipation into an accumulated Bellman residual. Third, an exact LQ performance-difference identity controls the objective gap and parameter errors by this residual. 
Consequently, the rate does not require a bounded value-function estimate, a uniform oscillation bound for the soft action-value function, or a ratio bound between policy densities.

\end{itemize}

\paragraph{Organization.}
Section~\ref{sec:related} reviews related work. Section~\ref{sec:lq_setup} formulates the unrestricted entropy-regularized LQ problem and proves optimality of a linear-Gaussian policy. Section~\ref{sec:wpg_lqg_reduction} defines the policy-space metric, derives WPGF, and proves exact closure of its direction on the linear-Gaussian policy class. Section~\ref{sec:geometry} establishes global well-posedness and global exponential convergence of the parameter flow, with explicit discussion of temperature and problem conditioning. The appendices compare policy geometries and give the complete proofs.

\section{Related Work}\label{sec:related}

\paragraph{Policy optimization for LQ control.}
The modern convergence theory of policy optimization for infinite-horizon LQ control begins with \citet{fazel2018lqrpg}, who show that the deterministic feedback-gain objective is nonconvex but satisfies a gradient-dominance property. This structure yields global guarantees for policy gradient, natural policy gradient, and Gauss-Newton methods. Subsequent work develops derivative-free guarantees \citep{malik2019derivativefree}, finite-horizon theory for noisy LQ control \citep{hambly2021lqr}, and model-free convergence and sample-complexity bounds \citep{mohammadi2022lqr}. These results primarily optimize a deterministic feedback gain and do not treat the full conditional action distribution as the policy variable.

For entropy-regularized LQ control, \citet{guo2026fast} study the
same discounted infinite-horizon model as our considered setting, but develop different algorithms
and a different convergence analysis. Their regularized policy-gradient
(RPG) and iterative policy-optimization (IPO) methods are formulated
directly in the linear-Gaussian policy parameters. After aligning sign and
step-size conventions, the RPG gain direction coincides with the WPG gain direction, but the covariance direction is different; the IPO covariance update is an exact one-step policy improvement, while the WPG covariance update is a Wasserstein gradient step. 

\citet{giegrich2024convergence} consider a further distinct setting:
finite-horizon continuous-time exploratory LQ control with
time-dependent linear-Gaussian policies in a noncoervive setting. Their method uses Fisher
geometry for the policy mean and Bures-Wasserstein geometry for the
covariance. The use of Bures-Wasserstein geometry is related to the
covariance component of the Gaussian \(2\)-Wasserstein metric, but their
algorithm combines two separate geometries on the Gaussian parameters;
it is not derived from a single statewise Wasserstein metric on the full
conditional action law.

\paragraph{Wasserstein policy optimization.}
Wasserstein policy methods update policies by transport in action space rather than by a Euclidean parameter change or a KL-based mirror step. Early work formulates policy optimization as a Wasserstein gradient flow \citep{zhang2018wgf}. Related developments include kernelized and computationally efficient Wasserstein natural gradients \citep{arbel2020kwng,moskovitz2020efficientwng}, optimal-transport trust-region methods \citep{terpin2022ottrpo}, metric-aware trust-region analysis \citep{song2024metrictr}, Wasserstein proximal updates \citep{zhu2026wasserstein}, and Wasserstein Policy Optimization \citep{pfau2025wpo}. These methods differ in how the policy-space direction is represented or approximated. Wasserstein Policy Optimization, for example, approximates a distributional policy flow in an actor-critic method and projects the direction onto a finite-dimensional actor. In the LQ model studied here, the linear-Gaussian class contains an unrestricted optimum and the WPG direction is tangent to this class, so the gain-covariance representation is exact and requires no projection.

\paragraph{General WPG convergence.}
The closest convergence result is \citet{zhu2026globalwpg}. They develop a general Bellman-based theory for WPG over continuous action distributions under boundedness and regularity assumptions and prove a discrete-time contraction up to a discretization error. Their continuous-time analysis is formulated for a sufficiently regular density flow, and the moving Gibbs family is controlled by the Holley-Stroock perturbation principle. The present analysis is complementary and specific to LQ control. Quadratic value functions yield an exact finite-dimensional parameter flow and permit a direct global well-posedness proof. The Gibbs law is Gaussian with sharp coefficient without a perturbation factor. The LQ identities also allow unbounded quadratic costs and give convergence statements directly for the feedback gain and the entropy-weighted covariance error. Thus the broader general theory and the sharper LQ analysis address different levels of generality.

\section{Entropy-Regularized Linear-Quadratic Control}
\label{sec:lq_setup}

\subsection{Problem setup}
\label{sec:problem_objective}

We consider the discounted stochastic linear system
\begin{equation}\label{eq:dynamics}
    x_{t+1}=Ax_t+Bu_t+w_t,\qquad t\ge0,
\end{equation}
where \(x_t\in\R^n\), \(u_t\in\R^m\), \(A\in\R^{n\times n}\), and
\(B\in\R^{n\times m}\). The noise sequence \((w_t)_{t\ge0}\) is i.i.d.,
independent of the initial state and of the past, has mean zero, and has
covariance \(W\succeq0\). The initial state \(x_0\sim \cD\) has full-rank second
moment
\[
    \Gamma_0:=\E[x_0x_0^\top]\succ0.
\]
The stage cost is
\begin{equation}\label{eq:stagecost}
    c(x,u)=x^\top Qx+u^\top Ru,
    \qquad
    Q\succeq0,\quad R\succ0.
\end{equation}

We use the cost-minimization convention. For a discount factor
\(\gamma\in(0,1)\), entropy temperature \(\tau>0\), and Markov policy
\(\pi(\cdot\mid x)\) with density, define
\begin{equation}\label{eq:Vpi}
    V^\pi(x)
    :=
    \E^\pi\!\left[
        \sum_{t=0}^{\infty}\gamma^t
        \bigl(c(x_t,u_t)+\tau\log\pi(u_t\mid x_t)\bigr)
        \,\middle|\,x_0=x
    \right],
    \qquad
    C(\pi):=\E_{x_0\sim \cD}[V^\pi(x_0)].
\end{equation}
The \(+\tau\log\pi\) term is a negative-entropy cost and therefore favors
randomized action laws under minimization. For comparisons beyond the Gaussian
class, let \(\Pi_{\mathrm{adm}}\) denote stationary randomized Markov policies
with densities such that \(V^\pi(x)\) is finite and the following discounted
transversality condition holds for every initial state:
\[
    \gamma^T\E_x^\pi[1+\|x_T\|^2]\longrightarrow0,
    \qquad x\in\R^n.
\]
This condition is automatic for the admissible linear-Gaussian policies considered
below. We write
\begin{equation}\label{eq:scalars}
    \mu:=\lambda_{\min}(\Gamma_0)>0,
    \qquad
    \lambda_R:=\lambda_{\min}(R)>0,
\end{equation}
for the two scalar problem constants that enter the convergence rate.
Throughout, \(\norm{\cdot}\) denotes the Euclidean norm for vectors and the
spectral norm for matrices.

We first introduce the linear-Gaussian family and its stability condition. We then prove, by a Bellman verification argument, that this family contains an optimal policy for the unrestricted control problem.

\subsection{Linear-Gaussian policies and admissibility}
\label{sec:lg_admissible}

We parameterize stochastic feedback policies by the linear-Gaussian family
\begin{equation}\label{eq:gaussian_policy}
    \pi_{K,\Sigma}(\cdot\mid x)=\cN(-Kx,\Sigma),
    \qquad
    K\in\R^{m\times n},\quad \Sigma\in\bbS^m_{++},
\end{equation}
where \(K\) parameterizes the state-dependent mean, \(\Sigma\) is the
state-independent action covariance, and \(\bbS^m_{++}\) denotes the cone of symmetric positive definite \(m\times m\) matrices. We write
\(V_{K,\Sigma}:=V^{\pi_{K,\Sigma}}\) and
\(C(K,\Sigma):=C(\pi_{K,\Sigma})\). Equivalently, under \(\pi_{K,\Sigma}\),
\[
    u_t=-Kx_t+\xi_t,
    \qquad
    \xi_t\sim\cN(0,\Sigma),
\]
with \((\xi_t)_{t\ge0}\) independent across time and independent of the system
noise.
The term \emph{linear-Gaussian policy} refers only to the conditional action law in~\eqref{eq:gaussian_policy}: the conditional law is Gaussian, its mean is linear in the state, and its covariance is state independent. Neither the initial state nor the disturbance is assumed Gaussian. Consequently, the state process need not be Gaussian.

For a gain \(K\), define
\[
    F_K:=A-BK,
    \qquad
    L_K:=Q+K^\top RK.
\]
The closed loop becomes
\begin{equation}
    x_{t+1}=F_Kx_t+\eta_t,
    \qquad
    \eta_t:=B\xi_t+w_t,
    \qquad
    \E[\eta_t]=0,
    \qquad
    \Cov(\eta_t)=\Omega_\Sigma:=B\Sigma B^\top+W.
\end{equation}

\begin{definition}[Admissibility]\label{def:admissible}
A pair \((K,\Sigma)\) is called \emph{admissible} if
\begin{equation}\label{eq:admissible}
    \Sigma\succ0,
    \qquad
    \rho(\sqrt{\gamma}F_K)<1.
\end{equation}
We denote the admissible set by \(\cA_{\mathrm{adm}}\), and we write
\[
    \cG_{\mathrm{adm}}
    :=
    \{\pi_{K,\Sigma}:(K,\Sigma)\in\cA_{\mathrm{adm}}\}
\]
for the corresponding class of admissible linear-Gaussian policies.
\end{definition}

For an admissible pair, define the discounted state-correlation matrix
\begin{equation}\label{eq:state-correlation}
    \Scorr_{K,\Sigma}
    :=
    \E^{K,\Sigma}\!\left[
        \sum_{t=0}^{\infty}\gamma^t x_tx_t^\top
    \right].
\end{equation}
Then \(\Scorr_{K,\Sigma}\) is finite and \(\Scorr_{K,\Sigma}\succeq\Gamma_0\succ0\),
because the \(t=0\) term equals \(\Gamma_0\). Throughout the paper, \(\Scorr\)
denotes discounted state-correlation matrices.

\begin{assumption}[Standing assumptions]\label{ass:standard}
Throughout the paper, we assume:
\begin{enumerate}[label=(\roman*)]
\item \((\sqrt{\gamma}A,\sqrt{\gamma}B)\) is stabilizable: there exists \(K\) such that \(\rho(\sqrt{\gamma}(A-BK))<1\).
\item \((Q^{1/2},\sqrt{\gamma}A)\) is detectable: there exists \(L\) such that \(\rho(\sqrt{\gamma}A-LQ^{1/2})<1\).
\item The initial linear-Gaussian policy \((K_0,\Sigma_0)\) has finite cost: \(C(K_0,\Sigma_0)<\infty\).
\end{enumerate}
\end{assumption}

These are the discounted analogues of the standard stabilizability and detectability assumptions for LQ control~\citep{anderson2007optimal}. Detectability rules out unstable closed loops with finite cost caused by modes unobserved by \(Q\); see Lemma~\ref{lem:finite_adm} in Appendix~\ref{app:stability}.

\subsection{Value representation and Bellman notation}
\label{sec:value_representation}

The next lemma gives the value and soft action-value functions in closed form. Its proof is a standard Lyapunov calculation and appears in Appendix~\ref{app:identities}.

\begin{lemma}[Quadratic value representation]\label{lem:value_representation}
For every admissible pair \((K,\Sigma)\), define
\begin{equation}\label{eq:PK_series}
    P_K
    :=
    \sum_{t=0}^{\infty}
    \gamma^t(F_K^t)^\top L_KF_K^t.
\end{equation}
Then \(P_K\) is the unique positive semidefinite solution of the discounted
Lyapunov equation
\begin{equation}\label{eq:PK_lyap}
    P_K=L_K+\gamma F_K^\top P_KF_K.
\end{equation}
Moreover, with
\begin{equation}\label{eq:notation_block}
    M_K:=R+\gamma B^\top P_KB,
    \qquad
    N_K:=\gamma B^\top P_KA,
    \qquad
    E_K:=M_KK-N_K,
\end{equation}
we have \(M_K\succeq R\succ0\), and the entropy-regularized value function is
\begin{equation}\label{eq:Vquad}
    V^{\pi_{K,\Sigma}}(x)=x^\top P_Kx+q_{K,\Sigma},
\end{equation}
where
\begin{equation}\label{eq:qKSigma}
    q_{K,\Sigma}
    =
    \frac{1}{1-\gamma}
    \left(
        \Tr(\Sigma M_K)
        -\frac{\tau}{2}
        \bigl(m+\log((2\pi)^m\det\Sigma)\bigr)
        +\gamma\Tr(WP_K)
    \right).
\end{equation}
The corresponding soft state-action value is
\begin{equation}\label{eq:Qquad}
\begin{aligned}
Q^{\pi_{K,\Sigma}}(x,u)
=
u^\top M_Ku
+2u^\top N_Kx
+\psi_{K,\Sigma}(x),
\end{aligned}
\end{equation}
where
\begin{equation}\label{eq:psiKSigma}
\psi_{K,\Sigma}(x)
=
x^\top\!\left(Q+\gamma A^\top P_KA\right)x
+\gamma\Tr(WP_K)
+\gamma q_{K,\Sigma}.
\end{equation}

Consequently,
\begin{equation}\label{eq:CKS}
    C(K,\Sigma)=\Tr(\Gamma_0P_K)+q_{K,\Sigma}.
\end{equation}
\end{lemma}

The matrix \(E_K\) is the feedback component of the Bellman residual: \(E_K=0\) precisely when the policy mean equals the current quadratic Bellman minimizer.

\subsection{Optimality of linear-Gaussian policies}
\begin{proposition}[Optimality over the unrestricted policy class]
\label{prop:optimizer}
Under Assumption~\ref{ass:standard}(i)-(ii), there is a linear-Gaussian policy
\[
\pi_\star(\cdot\mid x)=\cN(-K_\star x,\Sigma_\star)
\]
such that, for every \(\pi\in\Pi_{\mathrm{adm}}\) and every \(x\in\R^n\),
\begin{equation}\label{eq:gaussian_sufficiency}
\begin{aligned}
V^\pi(x)-V^{\pi_\star}(x)
&=
\tau\,\E_x^\pi\!\left[
\sum_{t=0}^{\infty}\gamma^t
\KL\!\left(
\pi(\cdot\mid x_t)
\,\middle\|
\pi_\star(\cdot\mid x_t)
\right)
\right]
\\
&\ge0.
\end{aligned}
\end{equation}
Hence, up to equality of the conditional densities almost everywhere, \(\pi_\star\) is the unique policy in \(\Pi_{\mathrm{adm}}\) that is optimal for every initial state. In particular,
\begin{equation}\label{eq:wlog_gaussian}
\min_{\pi\in\Pi_{\mathrm{adm}}}C(\pi)
=
\min_{(K,\Sigma)\in\cA_{\mathrm{adm}}}C(K,\Sigma)
=
C(K_\star,\Sigma_\star).
\end{equation}
Thus the policy optimization problem can be restricted to admissible linear-Gaussian policies without loss of optimality.

The optimal parameters satisfy
\[
E_{K_\star}=0,
\qquad
K_\star=M_{K_\star}^{-1}N_{K_\star},
\qquad
\Sigma_\star=\frac{\tau}{2}M_{K_\star}^{-1}.
\]
Moreover, \(P_\star:=P_{K_\star}\) is the unique stabilizing solution of the discounted algebraic Riccati equation
\[
P_\star
=
Q+\gamma A^\top P_\star A
-
\gamma^2 A^\top P_\star B
\left(R+\gamma B^\top P_\star B\right)^{-1}
B^\top P_\star A.
\]
\end{proposition}
The proof is given in Appendix~\ref{app:identities}.

Proposition~\ref{prop:optimizer} establishes the order of the analysis. Having first solved the unrestricted control problem, we may now solve the finite-dimensional learning problem
\begin{equation}\label{eq:gaussian_parameter_problem}
    \min_{(K,\Sigma)\in\cA_{\mathrm{adm}}} C(K,\Sigma).
\end{equation}
By~\eqref{eq:wlog_gaussian}, this restriction is without loss of optimality for the original control problem. The next section derives the WPG direction on \(\cG_{\mathrm{adm}}\) and shows that it has an exact representation in the gain and covariance. Thus optimality of the policy class and closure of the update are established as separate results.

\begin{remark}[Limit as the entropy temperature tends to zero]
The Riccati equation, and hence \(P_\star\), \(K_\star\), and
\(M_{K_\star}\), is independent of \(\tau\). The temperature parameter affects
only the optimal covariance:
\[
\Sigma_\star
=
\frac{\tau}{2}M_{K_\star}^{-1}.
\]
Consequently, \(\Sigma_\star\to0\) linearly as \(\tau\downarrow0\), and, for
every fixed \(x\),
\[
\pi_\star(\cdot\mid x)
=
\cN(-K_\star x,\Sigma_\star)
\longrightarrow
\delta_{-K_\star x}
\]
in \(W_2\). Thus the entropy-regularized optimal policy converges to the
deterministic discounted LQ controller as $\tau\downarrow0$.
\end{remark}

\section{Closure of the WPG Direction on the Linear-Gaussian Policy Class}
\label{sec:wpg_lqg_reduction}

By Proposition~\ref{prop:optimizer}, it suffices to optimize over \(\cG_{\mathrm{adm}}\). This section defines the policy-space metric, derives WPGF as its gradient-descent flow, and obtains its exact gain-covariance representation on this policy class.

Because the state space is unbounded and the costs have quadratic growth, the soft Bellman operator is naturally defined on
\[
    \Vtwo
    :=
    \left\{
        v:\R^n\to\R:
        \|v\|_{\Vtwo}
        :=
        \sup_{x\in\R^n}
        \frac{|v(x)|}{1+\|x\|^2}<\infty
    \right\}.
\]
For a Markov policy \(\pi\), define its closed-loop transition operator by
\[
    (\Pop_\pi v)(x)
    :=
    \E^\pi[v(x_1)\mid x_0=x],
\]
and write \(\Pop_{K,\Sigma}:=\Pop_{\pi_{K,\Sigma}}\).

For \(v\in\Vtwo\), define
\begin{equation}\label{eq:Qv}
    Q_v(x,u)
    :=
    c(x,u)+\gamma\,\E[v(Ax+Bu+w)].
\end{equation}
For a policy \(\pi\), write \(Q^\pi:=Q_{V^\pi}\) and define its soft
policy-evaluation operator by
\[
    (\cT_\pi v)(x)
    :=
    \int_{\R^m}
    \bigl(
        Q_v(x,u)+\tau\log\pi(u\mid x)
    \bigr)
    \pi(\dd u\mid x).
\]
Whenever the action minimization is finite, the soft Bellman operator is
\begin{equation}\label{eq:Tstar}
    (\cT_*v)(x)
    :=
    \inf_{\varrho(\cdot\mid x)}
    \int_{\R^m}
    \bigl(
        Q_v(x,u)+\tau\log\varrho(u\mid x)
    \bigr)
    \varrho(\dd u\mid x).
\end{equation}

\begin{definition}[Relative entropy, Fisher information, and LSI]
\label{def:kl_fisher_lsi}
Let \(p\) be a probability law on \(\R^m\) with positive smooth density, and let
\(\nu\ll p\). Define
\[
    \KL(\nu\|p)
    :=
    \int_{\R^m}
    \log\frac{\dd\nu}{\dd p}\,\dd\nu,
    \qquad
    \mathcal I(\nu\|p)
    :=
    \int_{\R^m}
    \left\|
        \nabla_u\log\frac{\dd\nu}{\dd p}
    \right\|^2\,\dd\nu.
\]
We use the convention \(\mathcal I(\nu\|p)=+\infty\) if \(\log(\dd\nu/\dd p)\) does not have a weak gradient with respect to \(u\). We say that \(p\) satisfies an \(\alpha\)-log-Sobolev
inequality if
\[
    \mathcal I(\nu\|p)\ge 2\alpha\,\KL(\nu\|p)
    \qquad
    \text{for all probability laws } \nu\ll p.
\]
\end{definition}

\subsection{Policy-space Wasserstein metric and gradient flow}
\label{sec:action_wpgf}

We now state the policy-space geometry used to define WPG. It is the
statewise 2-Wasserstein geometry used in the WPG literature, with the
statewise metric weighted by the current discounted state-occupancy measure
\citep{zhang2018wgf,pfau2025wpo,zhu2026globalwpg}. For a policy \(\pi\), let
\[
d_{\mathcal D}^{\pi}(A)
:=(1-\gamma)\sum_{t=0}^{\infty}
\gamma^t\,
\mathbb P_{\mathcal D}^{\pi}(x_t\in A),
\qquad
A\in\mathcal B(\R^n),
\]
be its normalized discounted state-occupancy measure.

\begin{definition}[Discounted-occupancy-weighted statewise Wasserstein metric]
\label{def:policy_wasserstein_metric}
Fix a probability measure \(\nu\) on the state space. For two policies whose
conditional action laws belong to \(\mathcal P_2(\R^m)\) and for which the
integral below is finite, define
\begin{equation}
\label{eq:policy_wasserstein_distance}
\mathsf W_{2,\nu}^2(\pi,\widetilde\pi)
:=
\frac{1}{1-\gamma}
\int_{\R^n}
W_2^2\!\left(
\pi(\cdot\mid x),
\widetilde\pi(\cdot\mid x)
\right)
\nu(\dd x).
\end{equation}
Policies that agree for \(\nu\)-almost every state are identified. At a base
policy \(\pi\) with positive smooth conditional densities, WPG uses
\(\nu=d_{\mathcal D}^{\pi}\).

To describe the corresponding tangent metric, begin with a test potential
\(\varphi\in C_c^\infty(\R^n\times\R^m)\), set
\(v=\nabla_u\varphi\), and define the mass-preserving tangent perturbation
\begin{equation}
\label{eq:policy_tangent_representation}
\xi_v(u\mid x)
:=
-\nabla_u\cdot\bigl(\pi(u\mid x)v(x,u)\bigr).
\end{equation}
The tangent velocity space is the closure of these action-gradient fields in
the norm below. Equivalently, each tangent perturbation is represented by its
minimum-\(L^2(\pi(\cdot\mid x))\)-norm action-gradient velocity. For tangent
perturbations \(\xi\) and \(\zeta\), with gradient representatives
\(v_\xi\) and \(v_\zeta\), define
\begin{equation}
\label{eq:policy_wasserstein_inner_product}
\begin{aligned}
\langle \xi,\zeta\rangle_{\pi,\mathsf W}
&:=
\frac{1}{1-\gamma}
\int_{\R^n}d_{\mathcal D}^{\pi}(\dd x)
\int_{\R^m}
\langle v_\xi(x,u),v_\zeta(x,u)\rangle
\pi(\dd u\mid x),\\
\|\xi\|_{\pi,\mathsf W}^2
&:=
\langle \xi,\xi\rangle_{\pi,\mathsf W}.
\end{aligned}
\end{equation}
\end{definition}

For fixed \(\nu\), \(\mathsf W_{2,\nu}\) is the direct integral of the
statewise 2-Wasserstein distances. WPG uses its tangent metric with the state
weight evaluated at the base policy. The factor \((1-\gamma)^{-1}\) converts
the normalized occupancy measure into the usual discounted sum. If
\(d_{\mathcal D}^{\pi}\) does not have full support, the tangent norm is
understood after identifying perturbations that agree
\(d_{\mathcal D}^{\pi}\)-almost everywhere in the state variable.

We next compute the first variation at an admissible linear-Gaussian policy.
Fix \((K,\Sigma)\in\cA_{\mathrm{adm}}\), write
\(\pi=\pi_{K,\Sigma}\), take
\(\varphi\in C_c^\infty(\R^n\times\R^m)\), and set
\[
    v=\nabla_u\varphi,
    \qquad
    T_{\varepsilon,x}(u)=u+\varepsilon v(x,u),
    \qquad
    \pi^\varepsilon(\cdot\mid x)
    =(T_{\varepsilon,x})_\#\pi(\cdot\mid x).
\]
Lemma~\ref{lem:conditional_first_variation} proves directly that this
transport path is well defined for all sufficiently small
\(\lvert\varepsilon\rvert\), that \(C(\pi^\varepsilon)\) is differentiable at
zero, and that
\begin{equation}
\label{eq:first_variation_main}
\begin{aligned}
DC(\pi)[\xi_v]
&=
\frac{1}{1-\gamma}
\int_{\R^n}d_{\mathcal D}^{\pi}(\dd x)
\int_{\R^m}
\bigl(Q^\pi(x,u)+\tau\log\pi(u\mid x)\bigr)
\xi_v(u\mid x)\,\dd u.
\end{aligned}
\end{equation}
Subtracting the conditional mean of the expression in parentheses gives the
usual centered first-variation representative and does not change the value
because \(\int\xi_v(u\mid x)\,\dd u=0\). Integration by parts gives
\begin{equation}
\label{eq:first_variation_transport_main}
\begin{aligned}
DC(\pi)[\xi_v]
&=
\frac{1}{1-\gamma}
\int_{\R^n}d_{\mathcal D}^{\pi}(\dd x)
\int_{\R^m}
\left\langle
\nabla_u\bigl(Q^\pi(x,u)+\tau\log\pi(u\mid x)\bigr),
 v(x,u)
\right\rangle
\pi(\dd u\mid x).
\end{aligned}
\end{equation}
The same lemma verifies the domination, differentiation of the discounted
series, Fubini interchange, and boundary terms used in these formulas.
Importantly, the perturbing policies \(\pi^\varepsilon\) need not be Gaussian;
thus the calculation identifies the policy-space gradient at a
linear-Gaussian policy before the closure result below is invoked.

\begin{proposition}[WPGF as a policy-space gradient flow]
\label{prop:wpg_metric_gradient}
At every admissible linear-Gaussian policy \(\pi=\pi_{K,\Sigma}\), the gradient
of \(C\) under the metric in
Definition~\ref{def:policy_wasserstein_metric} is
\begin{equation}
\label{eq:policy_space_gradient}
\operatorname{grad}_{\mathsf W}C(\pi)(u\mid x)
=
-\nabla_u\cdot
\left[
\pi(u\mid x)
\nabla_u\bigl(Q^\pi(x,u)+\tau\log\pi(u\mid x)\bigr)
\right].
\end{equation}
Consequently, the gradient-descent flow
\(\partial_t\pi_t=-\operatorname{grad}_{\mathsf W}C(\pi_t)\) is
\begin{equation}
\label{eq:wpgf}
\tag{WPGF}
\partial_t\pi_t(u\mid x)
=
\nabla_u\cdot
\left[
\pi_t(u\mid x)
\nabla_u\bigl(Q^{\pi_t}(x,u)+\tau\log\pi_t(u\mid x)\bigr)
\right].
\end{equation}
\end{proposition}

\begin{proof}
Fix \(\pi=\pi_{K,\Sigma}\) and set
\[
    g_\pi(x,u)
    :=
    \nabla_u\bigl(Q^\pi(x,u)+\tau\log\pi(u\mid x)\bigr).
\]
The function inside the action gradient is quadratic in \((x,u)\), so
\(g_\pi\) is affine. Admissibility gives a finite discounted second state
moment, and the conditional action law has finite second moments. Hence
\[
\frac{1}{1-\gamma}
\int d_{\mathcal D}^{\pi}(\dd x)
\int \|g_\pi(x,u)\|^2\pi(\dd u\mid x)
<\infty.
\]
Moreover, \(g_\pi\) belongs to the completed tangent velocity space in
Definition~\ref{def:policy_wasserstein_metric}. It is the action gradient of a
quadratic function, and smooth cutoffs in the state and action variables
approximate it in the metric norm; the details are given at the end of the
proof of Lemma~\ref{lem:conditional_first_variation}.

For every test velocity \(v=\nabla_u\varphi\),
Equation~\eqref{eq:first_variation_transport_main} gives
\[
DC(\pi)[\xi_v]
=
\left\langle
-\nabla_u\cdot(\pi g_\pi),
\xi_v
\right\rangle_{\pi,\mathsf W}.
\]
The test velocities are dense by definition, so the identity extends by
continuity to the tangent space. This proves
\eqref{eq:policy_space_gradient}; changing the sign gives~\eqref{eq:wpgf}.
The same discounted occupancy measure and normalization appear in the first
variation and in the metric, so no additional state-dependent factor enters
the WPG velocity.
\end{proof}

The metric determines the gradient for
\(d_{\mathcal D}^{\pi}(\dd x)\pi(\dd u\mid x)\)-almost every \((x,u)\). As in
the statewise WPG convention, we use the displayed Bellman field as the
representative at every state where it is defined. This gives the full-state
policy update used in the closure calculation below.

For each state \(x\), define the Gibbs density associated with the current
soft action-value function by
\begin{equation}
\label{eq:gibbs_general}
p_t(u\mid x)
:=
\frac{
\exp\bigl(-Q^{\pi_t}(x,u)/\tau\bigr)
}{
\displaystyle
\int_{\R^m}
\exp\bigl(-Q^{\pi_t}(x,a)/\tau\bigr)\,\dd a
}.
\end{equation}
Since the normalizing constant in~\eqref{eq:gibbs_general} is independent of
\(u\),
\[
\nabla_u
\bigl(
Q^{\pi_t}(x,u)+\tau\log\pi_t(u\mid x)
\bigr)
=
\tau
\nabla_u
\log
\frac{\pi_t(u\mid x)}{p_t(u\mid x)}.
\]
Consequently, WPGF has the equivalent Fokker-Planck representation
\begin{equation}
\label{eq:wpgf_re_form}
\partial_t\pi_t(u\mid x)
=
\tau\nabla_u\cdot
\left(
\pi_t(u\mid x)
\nabla_u
\log\frac{\pi_t(u\mid x)}{p_t(u\mid x)}
\right).
\end{equation}
The Gibbs law \(p_t(\cdot\mid x)\) is not fixed: it is determined by \(Q^{\pi_t}\), which depends on the current policy through the Bellman equation. Thus, at each time, \eqref{eq:wpgf_re_form} is the relative-entropy gradient flow toward the current Gibbs law, while that Gibbs law changes with the policy.

\subsection{Closure on the linear-Gaussian policy class}
\label{sec:gaussian_reduction}

We now evaluate the statewise WPG vector field on \(\cG_{\mathrm{adm}}\). Lemma~\ref{lem:value_representation} gives
\begin{equation}\label{eq:Qquad_reduction}
    Q^{\pi_{K,\Sigma}}(x,u)
    =
    u^\top M_Ku+2u^\top N_Kx+\psi_{K,\Sigma}(x),
\end{equation}
where \(\psi_{K,\Sigma}\) is independent of \(u\). Hence
\[
    \nabla_u Q^{\pi_{K,\Sigma}}(x,u)
    =
    2M_Ku+2N_Kx,
\]
which is affine in \(u\). The following proposition shows that the WPG vector field is tangent to the linear-Gaussian policy class and identifies its exact parameter representation.

\begin{proposition}[Exact closure on linear-Gaussian policies]
\label{prop:gaussian_reduction}
For every \((K,\Sigma)\in\cA_{\mathrm{adm}}\), the right-hand side of the statewise WPG equation evaluated at \(\pi_{K,\Sigma}\) is tangent to \(\cG_{\mathrm{adm}}\). More precisely, let \(D\pi_{K,\Sigma}[\dot K,\dot\Sigma]\) denote the directional derivative of the Gaussian density with respect to its parameters. The unique parameter velocity \((\dot K,\dot\Sigma)\) satisfying
\begin{equation}\label{eq:tangent_identity}
\begin{aligned}
&D\pi_{K,\Sigma}[\dot K,\dot\Sigma](u\mid x)\\
&\qquad=
\nabla_u\cdot\!\left[
\pi_{K,\Sigma}(u\mid x)
\nabla_u\!\left(
Q^{\pi_{K,\Sigma}}(x,u)
+
\tau\log\pi_{K,\Sigma}(u\mid x)
\right)
\right]
\end{aligned}
\end{equation}
for every \(x\) and \(u\) is
\begin{equation}\label{eq:KSigma_vector_field}
    \dot K=-2E_K,
    \qquad
    \dot\Sigma=-2M_K\Sigma-2\Sigma M_K+2\tau I.
\end{equation}
Consequently, on every interval on which \((K_t,\Sigma_t)\) is an admissible solution of
\begin{equation}\label{eq:KSigma_ode}
\begin{aligned}
    \dot K_t&=2(N_t-M_tK_t)=-2E_t,\\
    \dot\Sigma_t&=-2M_t\Sigma_t-2\Sigma_tM_t+2\tau I,
\end{aligned}
\end{equation}
the curve \(\pi_t=\pi_{K_t,\Sigma_t}\) is a linear-Gaussian solution of the statewise WPG equation~\eqref{eq:wpgf}.
\end{proposition}

\begin{remark}[Meaning of the closure result]
Proposition~\ref{prop:gaussian_reduction} identifies the statewise WPG vector field exactly on the policy class that is sufficient for optimality. The parameter ODE is therefore the WPG dynamics on \(\cG_{\mathrm{adm}}\), rather than a projection or approximation of the policy-space direction.
\end{remark}

The proof is given in Appendix~\ref{app:flow}. Proposition~\ref{prop:gaussian_reduction} defines the closed linear-Gaussian WPG parameter flow analyzed below.

Here all Bellman quantities are evaluated at the current gain \(K_t\). Namely,
with
\[
    P_t:=P_{K_t},
    \qquad
    M_t:=R+\gamma B^\top P_tB,
    \qquad
    N_t:=\gamma B^\top P_tA,
    \qquad
    E_t:=M_tK_t-N_t,
\]
the matrix \(P_t\) is the unique positive semidefinite solution of
\[
    P_t
    =
    Q+K_t^\top RK_t
    +
    \gamma(A-BK_t)^\top P_t(A-BK_t).
\]

Consequently, the Gibbs law in~\eqref{eq:gibbs_general} is also Gaussian:
\begin{equation}\label{eq:gibbs_policy_lq}
p_t(\cdot\mid x)
=
\cN(-M_t^{-1}N_tx,\Xi_t),
\qquad
\Xi_t:=\frac{\tau}{2}M_t^{-1}.
\end{equation}

The gain equation follows the feedback residual \(E_t=M_tK_t-N_t\). The covariance equation combines contraction generated by the quadratic action-cost matrix \(M_t\) with the additive entropy term \(2\tau I\).

\begin{remark}[Relation to other policy geometries]
For comparison, Table~\ref{tab:geometry_comparison} in
Appendix~\ref{app:geometry_comparison} summarizes the joint gain and
covariance updates induced by WPGF and other policy geometries. 
\end{remark}

\section{Global Convergence Analysis}
\label{sec:geometry}

We now analyze the linear-Gaussian parameter flow. A Bellman resolvent identity converts the relative Fisher-information dissipation at each state into discounted value decrease. The Gaussian Bellman residual and the LQ performance-difference identity then compare this decrease with the global objective gap. All estimates are expressed through the matrices \(P_K\), \(M_K\), \(E_K\), and \(\Scorr_{K,\Sigma}\); no bounded value-function or soft-\(Q\) estimate is used.

\subsection{Value descent through the Bellman resolvent}
\label{sec:value_descent}

Because the parameter ODE is the exact closure of the statewise Wasserstein policy gradient on \(\cG_{\mathrm{adm}}\), its descent is most naturally expressed at the value-function level. Let
\[
    V_t:=V^{\pi_t},
    \qquad
    \Pop_t:=\Pop_{\pi_t}=\Pop_{K_t,\Sigma_t}.
\]
Define
\begin{equation}\label{eq:gt_definition}
    g_t(x)
    :=
    \tau^2
    \mathcal I
    \bigl(
        \pi_t(\cdot\mid x)
        \|p_t(\cdot\mid x)
    \bigr),
\end{equation}
where \(p_t\) is the current Gibbs policy in \eqref{eq:gibbs_policy_lq}. Then
\(g_t(x)\ge0\), and the following identity converts local action-space
dissipation into value decrease.

\begin{lemma}[Value resolvent identity]
\label{lem:resolvent}
Along the linear-Gaussian parameter flow in Proposition~\ref{prop:gaussian_reduction},
\begin{equation}\label{eq:resolvent_identity}
    (I-\gamma\Pop_t)\frac{\dd}{\dd t}V_t
    =
    -g_t.
\end{equation}
Consequently,
\begin{equation}\label{eq:value_monotonicity}
    \frac{\dd}{\dd t}V_t(x)
    =
    -
    \sum_{i=0}^{\infty}
    \gamma^i(\Pop_t^i g_t)(x)
    \le
    -g_t(x)
    \le0,
    \qquad
    \forall x\in\R^n.
\end{equation}
\end{lemma}

The proof is given in Appendix~\ref{app:flow}.

This identity propagates the action-space dissipation at each state through the closed-loop dynamics and converts it into monotone value improvement.

\subsection{Bellman identities}
\label{sec:bellman_geometry}

For an admissible \((K,\Sigma)\), define
\[
\Xi_K:=\frac{\tau}{2}M_K^{-1},
\qquad
p_K(\cdot\mid x):=\cN(-M_K^{-1}N_Kx,\Xi_K).
\]
The next two identities are the LQ-specific inputs to the convergence proof; both are proved in Appendix~\ref{app:identities}.

\begin{lemma}[Bellman residual and Gaussian LSI]
\label{lem:bellman_lsi}
Fix an admissible \((K,\Sigma)\). The Bellman residual
\[
    R_{K,\Sigma}(x)
    :=
    V_{K,\Sigma}(x)
    -
    (\cT_*V_{K,\Sigma})(x)
\]
satisfies
\begin{equation}\label{eq:bellman_kl}
\begin{aligned}
    R_{K,\Sigma}(x)
    &=
    \tau\,
    \KL\!\left(
        \pi_{K,\Sigma}(\cdot\mid x)
        \,\middle\|\,
        p_K(\cdot\mid x)
    \right) \\
    &=
    x^\top E_K^\top M_K^{-1}E_Kx
    +
    \frac{\tau}{2}
    \left(
        \Tr(\Xi_K^{-1}\Sigma)
        -
        m
        -
        \log\det(\Xi_K^{-1}\Sigma)
    \right).
\end{aligned}
\end{equation}
Moreover, for every \(x\in\R^n\), define
\[
    \alpha_K
    :=
    \lambda_{\min}(\Xi_K^{-1})
    =
    \frac{2\lambda_{\min}(M_K)}{\tau},
    \qquad
    \underline\alpha
    :=
    \frac{2\lambda_R}{\tau}.
\]
Then \(\alpha_K\ge\underline\alpha\), and the Gibbs law
\(p_K(\cdot\mid x)\) satisfies
\begin{equation}\label{eq:lq_lsi}
\begin{aligned}
    \mathcal I(\nu\|p_K(\cdot\mid x))
    &\ge
    2\alpha_K
    \KL\!\left(
        \nu
        \,\middle\|\,
        p_K(\cdot\mid x)
    \right)\\
    &\ge
    2\underline\alpha
    \KL\!\left(
        \nu
        \,\middle\|\,
        p_K(\cdot\mid x)
    \right)
\end{aligned}
\end{equation}
for every probability law
\(\nu\ll p_K(\cdot\mid x)\). The coefficient \(\alpha_K\) is sharp for this Gaussian law. Equivalently, in the convention
\(\KL(\nu\|p_K)\le(C_{\mathrm{LSI}}(p_K)/2)\mathcal I(\nu\|p_K)\), the sharp constant is
\[
    C_{\mathrm{LSI}}(p_K)
    =
    \frac{\tau}{2\lambda_{\min}(M_K)}.
\]
\end{lemma}

The sharp coefficient follows directly from the covariance
\(\Xi_K=(\tau/2)M_K^{-1}\). The convergence proof uses the uniform lower bound
\(\underline\alpha=2\lambda_R/\tau\), obtained from
\(M_K\succeq R\); this lower bound need not be sharp for a fixed \(K\). In the general WPG analysis of \citet{zhu2026globalwpg}, the moving Gibbs law is controlled as a bounded perturbation of a Gaussian, leading to a coefficient of the form
\[
    \alpha_{\mathrm{gen}}
    =
    \frac{\beta}{\tau}
    \exp\!\left(-\frac{2U_{\max}}{\tau}\right).
\]
In the present LQ model, no perturbation step is needed. Since WPG dissipates \(\tau^2 \mathcal I(\pi\|p)\) and the Bellman residual is \(R=\tau\KL(\pi\|p)\), the uniform bound gives
\[
    \tau^2 \mathcal I(\pi\|p)
    \ge
    4\lambda_R R.
\]
Thus this part of the contraction estimate contains no factor of the form \(\exp(-c/\tau)\).

\begin{lemma}[Performance-difference identity]
\label{lem:perf}
Let \((K,\Sigma)\) and \((K',\Sigma')\) be admissible, and define
\[
    \Delta K:=K'-K,
    \qquad
    f_K(\Sigma)
    :=
    \frac{\tau}{2(1-\gamma)}\log\det\Sigma
    -
    \frac{1}{1-\gamma}\Tr(M_K\Sigma).
\]
Then
\begin{equation}\label{eq:perf_diff}
\begin{aligned}
    C(K',\Sigma')-C(K,\Sigma)
    &=
    \Tr\!\left(
        \Scorr_{K',\Sigma'}
        \left(
            \Delta K^\top M_K\Delta K
            +
            2\Delta K^\top E_K
        \right)
    \right) \\
    &\qquad
    +
    f_K(\Sigma)
    -
    f_K(\Sigma').
\end{aligned}
\end{equation}
Consequently, \(C\) is differentiable on the admissible set, with
\begin{equation}\label{eq:gradients}
    \nabla_K C(K,\Sigma)
    =
    2E_K\Scorr_{K,\Sigma},
    \qquad
    \nabla_\Sigma C(K,\Sigma)
    =
    \frac{1}{1-\gamma}
    \left(
        M_K-\frac{\tau}{2}\Sigma^{-1}
    \right).
\end{equation}
\end{lemma}

Together, these identities connect the stationary conditions of the reduced
flow to the global objective gap.

\subsection{Global well-posedness}
\label{sec:global_wellposedness}
To prove global existence, we must show that the trajectory cannot lose discounted stability or approach a singular covariance. The next proposition identifies admissibility with finite cost and shows that every finite-cost sublevel set is a compact subset of the admissible set.

\begin{proposition}[Finite cost, stability, and compact sublevels]
\label{prop:stability_compactness}
Under Assumption~\ref{ass:standard}, for every \(K\in\R^{m\times n}\) and
\(\Sigma\in\bbS_{++}^m\),
\[
    C(K,\Sigma)<\infty
    \quad\Longleftrightarrow\quad
    \rho\!\left(\sqrt{\gamma}\,F_K\right)<1.
\]
Moreover, for every \(c\in\R\), the sublevel set
\[
    \mathcal S_c
    :=
    \left\{
        (K,\Sigma)\in\cA_{\mathrm{adm}}
        :
        C(K,\Sigma)\le c
    \right\}
\]
is compact. In particular, it has positive distance from the boundary of
\(\cA_{\mathrm{adm}}\).
\end{proposition}

The proof is given in Appendix~\ref{app:stability}. Detectability is
essential: without it, a discounted-unstable mode that is invisible to the state cost may still have finite objective value.

\begin{lemma}[Global well-posedness of the closed parameter flow]
\label{lem:global}
Under Assumption~\ref{ass:standard}, the ODE \eqref{eq:KSigma_ode} with initial
condition \((K_0,\Sigma_0)\) has a unique solution
\(\{(K_t,\Sigma_t)\}_{t\ge0}\) for all \(t\ge0\). The trajectory remains
admissible and stays in the sublevel set
\[
    \mathcal S_{C(K_0,\Sigma_0)}
    :=
    \left\{
        (K,\Sigma)\in\cA_{\mathrm{adm}}:
        C(K,\Sigma)\le C(K_0,\Sigma_0)
    \right\}.
\]
Thus the closed WPG parameter flow is well defined for all \(t\ge0\).
\end{lemma}

The proof is given in Appendix~\ref{app:stability}. It is a finite-dimensional ODE argument based on descent of the cost and compactness of finite-cost sublevel sets.

\subsection{Global convergence and temperature dependence}

The next theorem gives the global objective rate. Two corollaries make its parameter and temperature dependence explicit.

\begin{theorem}[Global exponential convergence]
\label{thm:main}
Under Assumption~\ref{ass:standard}, let \(\{(K_t,\Sigma_t)\}_{t\ge0}\) be the WPG parameter trajectory from an admissible initialization \((K_0,\Sigma_0)\), and let \((K_\star,\Sigma_\star)\) be the optimizer in Proposition~\ref{prop:optimizer}. Define
\[
\Scorr_\star:=\Scorr_{K_\star,\Sigma_\star},
\qquad
\lambda_\tau:=\frac{4\mu\lambda_R}{\norm{\Scorr_\star}}>0.
\]
Then, along the WPG parameter trajectory,
\begin{equation}
\label{eq:pl_differential}
-\frac{\dd}{\dd t}C(K_t,\Sigma_t)
\ge
\lambda_\tau\bigl(C(K_t,\Sigma_t)-C(K_\star,\Sigma_\star)\bigr).
\end{equation}
Consequently, for every \(t\ge0\),
\begin{equation}
\label{eq:main_rate}
C(K_t,\Sigma_t)-C(K_\star,\Sigma_\star)
\le
\exp(-\lambda_\tau t)
\bigl(C(K_0,\Sigma_0)-C(K_\star,\Sigma_\star)\bigr).
\end{equation}
Equivalently, if \(0<\varepsilon<C(K_0,\Sigma_0)-C(K_\star,\Sigma_\star)\), then the objective gap is at most \(\varepsilon\) once
\begin{equation}
\label{eq:time_to_accuracy}
    t
    \ge
    \frac{\norm{\Scorr_\star}}{4\mu\lambda_R}
    \log\!\left(
        \frac{C(K_0,\Sigma_0)-C(K_\star,\Sigma_\star)}{\varepsilon}
    \right).
\end{equation}
\end{theorem}

\begin{corollary}[Gain and entropy-weighted covariance convergence]
\label{cor:parameter_rate}
Define
\[
    \Phi(X):=\Tr(X)-m-\log\det X,
    \qquad X\in\bbS_{++}^m.
\]
Along the trajectory in Theorem~\ref{thm:main},
\begin{equation}
\label{eq:parameter_rate}
\begin{aligned}
&\mu\lambda_R\norm{K_t-K_\star}_F^2
+
\frac{\tau}{2(1-\gamma)}
\Phi\!\left(
    \Sigma_\star^{-1/2}\Sigma_t\Sigma_\star^{-1/2}
\right) \\
&\qquad\le
\exp(-\lambda_\tau t)
\bigl(C(K_0,\Sigma_0)-C(K_\star,\Sigma_\star)\bigr).
\end{aligned}
\end{equation}
In particular, the squared gain error and the entropy-weighted relative covariance error decay with the same exponential exponent as the objective gap.
\end{corollary}

\begin{proof}
Apply the performance-difference identity in Lemma~\ref{lem:perf} with the optimal pair \((K_\star,\Sigma_\star)\) as the reference policy. Since \(E_{K_\star}=0\),
\[
\begin{aligned}
C(K,\Sigma)-C(K_\star,\Sigma_\star)
&=
\Tr\!\left(
\Scorr_{K,\Sigma}
(K-K_\star)^\top M_{K_\star}(K-K_\star)
\right) \\
&\quad+
\frac{\tau}{2(1-\gamma)}
\Phi\!\left(
\Sigma_\star^{-1/2}\Sigma\Sigma_\star^{-1/2}
\right).
\end{aligned}
\]
Using \(\Scorr_{K,\Sigma}\succeq\Gamma_0\succeq\mu I\) and \(M_{K_\star}\succeq R\succeq\lambda_R I\), and then applying Theorem~\ref{thm:main}, proves the claim.
\end{proof}

\begin{corollary}[Temperature dependence of the exponential rate]
\label{cor:temperature_rate}
Let
\[
    F_\star:=A-BK_\star,
    \qquad
    M_\star:=M_{K_\star}.
\]
The matrices \(K_\star\), \(P_\star\), \(M_\star\), and \(F_\star\) do not depend on \(\tau\). Moreover,
\begin{equation}
\label{eq:Sstar_affine_tau}
    \Scorr_\star(\tau)
    =
    \Scorr_\star^{(0)}
    +
    \tau\Scorr_\star^{(1)},
\end{equation}
where \(\Scorr_\star^{(0)}\) and \(\Scorr_\star^{(1)}\) are the unique positive semidefinite solutions of
\begin{align*}
\Scorr_\star^{(0)}
&=
\Gamma_0
+
\gamma F_\star\Scorr_\star^{(0)}F_\star^\top
+
\frac{\gamma}{1-\gamma}W,\\
\Scorr_\star^{(1)}
&=
\gamma F_\star\Scorr_\star^{(1)}F_\star^\top
+
\frac{\gamma}{2(1-\gamma)}
B M_\star^{-1}B^\top.
\end{align*}
Consequently,
\begin{equation}
\label{eq:lambda_tau_lower}
    \lambda_\tau
    \ge
    \frac{4\mu\lambda_R}
    {\norm{\Scorr_\star^{(0)}}+\tau\norm{\Scorr_\star^{(1)}}},
\end{equation}
and
\[
    \lim_{\tau\downarrow0}\lambda_\tau
    =
    \frac{4\mu\lambda_R}{\norm{\Scorr_\star^{(0)}}}>0.
\]
Thus, for each fixed LQ problem, the exponent has the displayed positive limit. The time constant \(\norm{\Scorr_\star(\tau)}/(4\mu\lambda_R)\) in \eqref{eq:time_to_accuracy} is at most affine in \(\tau\) and converges to \(\norm{\Scorr_\star^{(0)}}/(4\mu\lambda_R)\) as \(\tau\downarrow0\).
\end{corollary}

\begin{proof}
Proposition~\ref{prop:optimizer} shows that the Riccati equation, and hence \(K_\star\), \(P_\star\), \(M_\star\), and \(F_\star\), is independent of \(\tau\), while
\[
    \Sigma_\star(\tau)=\frac{\tau}{2}M_\star^{-1}.
\]
For any admissible \((K,\Sigma)\), the discounted state-correlation matrix satisfies
\[
    \Scorr_{K,\Sigma}
    =
    \Gamma_0
    +
    \gamma F_K\Scorr_{K,\Sigma}F_K^\top
    +
    \frac{\gamma}{1-\gamma}
    (B\Sigma B^\top+W).
\]
Substituting the optimal pair and the formula for \(\Sigma_\star(\tau)\) gives \eqref{eq:Sstar_affine_tau} and the two displayed Lyapunov equations. The norm bound follows from the triangle inequality, and the limit follows from continuity of the spectral norm.
\end{proof}

\begin{remark}[Interpretation of the rate bound]
\label{rem:rate_interpretation}
Corollary~\ref{cor:temperature_rate} isolates the temperature dependence for a fixed LQ problem and a fixed choice of coordinates. Its magnitude continues to reflect the usual conditioning of the control problem: discounting and the optimal closed-loop stability margin enter through \(\norm{\Scorr_\star}\), initial-state excitation enters through \(\mu\), and action-cost curvature enters through \(\lambda_R\). State and action rescalings change these quantities as well. The result therefore removes the additional exponential low-temperature factor while preserving the familiar instance dependence.
\end{remark}

\begin{remark}[Interpretation of the covariance estimate]
\label{rem:covariance_interpretation}
Set
\[
X_t:=\Sigma_\star^{-1/2}\Sigma_t\Sigma_\star^{-1/2}.
\]
The covariance quantity in~\eqref{eq:parameter_rate} is \(\tau\Phi(X_t)\), up to the fixed factor \(2(1-\gamma)\), and is exactly the covariance contribution to the entropy-regularized objective. Since \(\Phi(X)\) is locally equivalent to \(\norm{X-I}_F^2\) near \(I\), this yields exponential convergence of the relative covariance for every fixed \(\tau>0\). On a fixed spectral neighborhood of \(I\), an unweighted bound for \(\norm{X_t-I}_F^2\) has a prefactor proportional to \(1/\tau\).
\end{remark}

\begin{proof}[Proof sketch of Theorem~\ref{thm:main}]
The proof has two quantitative steps. First, the Bellman resolvent identity and the Gaussian log-Sobolev inequality convert Wasserstein dissipation into an accumulated soft Bellman residual. Second, the LQ performance-difference identity compares this residual with the global objective gap.

Let
\[
\Scorr_t:=\Scorr_{K_t,\Sigma_t},
\qquad
R_t(x):=V_t(x)-(\cT_*V_t)(x),
\]
and define
\[
\mathcal R_t
:=
\E_{x_0\sim\cD}
\left[
\bigl((I-\gamma\Pop_{\pi_t})^{-1}R_t\bigr)(x_0)
\right].
\]
Lemma~\ref{lem:resolvent} expresses objective dissipation as the discounted accumulation of relative Fisher information under the current policy. Lemma~\ref{lem:bellman_lsi} gives the sharp coefficient
\(\alpha_{K_t}=2\lambda_{\min}(M_t)/\tau\), its uniform lower bound
\(2\lambda_R/\tau\), and the identity
\(R_t=\tau\KL(\pi_t\|p_t)\). Hence the explicit factors of \(\tau\) cancel and
\[
-\frac{\dd}{\dd t}C(K_t,\Sigma_t)
\ge
4\lambda_R\mathcal R_t.
\]

The Gaussian KL formula yields
\[
\mathcal R_t
=
\Tr\!\left(
\Scorr_tE_t^\top M_t^{-1}E_t
\right)
+
B_t,
\qquad
B_t:=f_{K_t}(\Xi_t)-f_{K_t}(\Sigma_t)\ge0.
\]
The LQ performance-difference identity gives
\[
C(K_t,\Sigma_t)-C(K_\star,\Sigma_\star)
\le
\Tr\!\left(
\Scorr_\star E_t^\top M_t^{-1}E_t
\right)
+
B_t.
\]
Using
\[
\Scorr_t\succeq\Gamma_0\succeq\mu I,
\qquad
\Scorr_\star\preceq\norm{\Scorr_\star}I,
\]
and \(B_t\ge0\), we obtain
\[
\mathcal R_t
\ge
\frac{\mu}{\norm{\Scorr_\star}}
\bigl(C(K_t,\Sigma_t)-C(K_\star,\Sigma_\star)\bigr).
\]
Combining these estimates proves~\eqref{eq:pl_differential}, and Gronwall's inequality gives~\eqref{eq:main_rate}. Appendix~\ref{app:main} contains the complete proof.
\end{proof}

\section{Conclusion}
\label{sec:conclusion}

We studied entropy-regularized discounted LQ control in two stages. A Bellman verification argument first shows that a linear-Gaussian policy is optimal over the unrestricted admissible stationary policy class. The policy optimization problem can therefore be restricted to the gain and covariance without loss of optimality. We then stated the discounted-occupancy-weighted statewise 2-Wasserstein metric explicitly and derived WPGF as the negative gradient flow of the control objective.

At every admissible linear-Gaussian policy, the WPG direction is tangent to the same class and is represented exactly by the feedback-gain and covariance velocities. The resulting finite-dimensional ODE generates a linear-Gaussian solution of WPGF without projection and is globally well posed from every admissible initialization. This is the well-posedness result needed for the parameter-flow analysis; it does not rely on a uniqueness statement for every solution of the unrestricted nonlinear policy equation. The parameter equations also clarify the comparison with Fisher--Rao flow: the two gain update formulas coincide when the action covariance is the identity, whereas their covariance update laws are fundamentally different.

The objective gap and squared gain error converge exponentially. The sharp Gaussian log-Sobolev coefficient of the current Gibbs law is \(2\lambda_{\min}(M_K)/\tau\), and its uniform lower bound \(2\lambda_R/\tau\) cancels the explicit temperature factor in the Fisher-information dissipation. Consequently, for each fixed LQ problem and action metric, the proved exponent has a positive limit as \(\tau\downarrow0\) and contains no perturbative \(\exp(-c/\tau)\) factor. Its magnitude continues to reflect the usual LQ conditioning through discounting, closed-loop stability, initial-state excitation, action-cost scaling, and coordinate choice. Covariance convergence is controlled in the entropy-weighted relative error; an unweighted relative covariance estimate can have a prefactor proportional to \(1/\tau\). 

\appendix

\section{Comparison with Other Policy Geometries}
\label{app:geometry_comparison}

Table~\ref{tab:geometry_comparison} writes several LQ control policy-gradient directions in a common notation. Gain-only methods optimize \(K\) while fixing or omitting the exploration covariance. In the joint linear-Gaussian setting, regularized policy gradient (RPG) and the Fisher-Rao natural gradient for the full linear-Gaussian policy both precondition the covariance gradient on the left and right, but their gain components differ: only the Fisher-Rao gain direction is premultiplied by \(\Sigma\). The closed WPG parameter equations instead follow directly from the statewise Wasserstein direction of the conditional action distributions.

\begin{table}[H]
\centering
\caption{Descent directions for the infinite-horizon discounted (entropy-regularized) LQ control problem.}
\label{tab:geometry_comparison}
\begingroup
\scriptsize
\setlength{\tabcolsep}{2.7pt}
\renewcommand{\arraystretch}{1.16}
\begin{tabular}{
@{}
>{\raggedright\arraybackslash}p{0.135\linewidth}
>{\raggedright\arraybackslash}p{0.20\linewidth}
>{\centering\arraybackslash}p{0.14\linewidth}
>{\centering\arraybackslash}p{0.235\linewidth}
>{\raggedright\arraybackslash}p{0.21\linewidth}
@{}
}
\toprule
Method
&
Geometry / preconditioner
&
Gain direction
&
Covariance direction
&
Theory / distinction
\\
\midrule
\multicolumn{5}{@{}l}{\emph{Deterministic-policy methods for standard LQ control: only \(K\) is optimized}}
\\[0.15em]
PG \citep{fazel2018lqrpg}
&
Euclidean in \(K\)
&
\(-2E_K\Scorr\)
&
\(0\)
&
Gain-only LQ control baseline; covariance is not learned.
\\
\addlinespace[0.35em]
Gain NPG \citep{fazel2018lqrpg}
&
State-correlation preconditioning in \(K\)
&
\(-2E_K\)
&
\(0\)
&
Standard LQ control NPG.
\\
\midrule
\multicolumn{5}{@{}l}{\emph{Entropy-regularized LQ control over joint linear-Gaussian policies: both \(K\) and \(\Sigma\) are optimized}}
\\[0.15em]
Euclidean GF
&
Product Euclidean geometry on \((K,\Sigma)\)
&
\(-2E_K\Scorr\)
&
\(-\dfrac{1}{1-\gamma}G_\Sigma\)
&
Formal Euclidean gradient flow for the joint parameter \((K,\Sigma)\); it does not preserve the geometry of positive-definite covariance matrices.
\\
\addlinespace[0.35em]
RPG \citep{guo2026fast}
&
Separate gain and covariance preconditioners
&
\(-2E_K\)
&
\(-\dfrac{\beta}{1-\gamma}\Sigma G_\Sigma\Sigma\)
&
Discrete global linear rate under its stated conditions; not the full
linear-Gaussian Fisher-Rao NPG.
\\
\addlinespace[0.35em]
Full Fisher-Rao NPG 
&
Fisher-Rao geometry of the complete conditional linear-Gaussian family
&
\(-2\Sigma E_K\)
&
\(-2\Sigma G_\Sigma\Sigma\)
&
Natural gradient for the full conditional linear-Gaussian policy; its gain component differs from RPG and WPGF.
\\
\addlinespace[0.35em]
\textbf{WPGF (ours)}
&
Conditional action-space \(W_2\); on the linear-Gaussian family, the induced metric is Euclidean for the mean and the \(W_2\) metric for the covariance
&
\(-2E_K\)
&
\(-2(G_\Sigma\Sigma+\Sigma G_\Sigma)\)
&
\textbf{Global well-posedness and an explicit global exponential objective
rate.}
\\
\bottomrule
\end{tabular}
\parbox{\linewidth}{
\vspace{3pt}\scriptsize
\emph{Notation and conventions.}
\(\Scorr:=\Scorr_{K,\Sigma}\) and
\(G_\Sigma:=M_K-(\tau/2)\Sigma^{-1}\). All rows use the cost-minimization sign
convention. PG and gain NPG freeze \(\Sigma\). For RPG, the displayed vector
field is the infinitesimal direction of the Guo-Li-Xu iteration with
\(\eta_1=h\), \(\eta_2=\beta h\), and \(h\downarrow0\); thus \(\beta>0\)
records the covariance-to-gain step-size ratio, while its stated guarantee is
for the discrete iterates. The Fisher-Rao row uses the discounted Fisher metric.
}
\endgroup
\end{table}

These directions are not interchangeable. Gain-only PG and NPG do not define a joint flow on \((K,\Sigma)\); RPG is a discrete method with separately chosen gain and covariance step-size scalings; and Euclidean and full Fisher-Rao flows use different joint policy geometries. The main theorem therefore analyzes the exact gain-covariance dynamics induced by WPG, including global well-posedness and the explicit global convergence.

\paragraph{Construction of Figure~\ref{fig:ambient_dimension}.}
Let the task-relevant action subspace have dimension \(r=8\). In an eigenbasis of \(M_\star\), set the optimal covariance eigenvalues to
\[
    \xi_i^\star=10^{-(i-1)/(r-1)},
    \qquad i=1,\ldots,r,
\]
and set every redundant-coordinate eigenvalue to \(10^{-6}\). Since \(\xi_i^\star=\tau/(2m_i)\), these values determine the corresponding eigenvalues \(m_i\) of \(M_\star\). With the rescaled time \(s=\tau t\), the local WPG gain error in mode \(i\) is
\[
    e_i^{\mathrm W}(s)=\exp(-s/\xi_i^\star).
\]
For the Fisher-Rao flow, write \(q_i=\sigma_i/\xi_i^\star\). To first order in the gain error, the modal covariance and gain equations are
\[
    \frac{\dd q_i}{\dd s}=q_i(1-q_i),
    \qquad
    \frac{\dd e_i^{\mathrm{FR}}}{\dd s}=-q_i e_i^{\mathrm{FR}},
\]
so, from \(q_i(0)=q_{i,0}\) and \(e_i^{\mathrm{FR}}(0)=1\),
\[
    e_i^{\mathrm{FR}}(s)
    =
    \frac{1}{1+q_{i,0}(e^s-1)}.
\]
Figure~\ref{fig:ambient_dimension} plots the worst error over the eight relevant modes. For each ambient dimension \(m\), the common isotropic initialization is \(\Sigma_0=c_mI\), where \(mc_m=10\Tr(\Sigma_\star)\); the trace includes the \(m-r\) redundant-coordinate eigenvalues. Hence \(q_{i,0}=c_m/\xi_i^\star\) decreases as redundant coordinates are added, which slows Fisher-Rao but leaves WPG unchanged.

\section{Linear-Quadratic Identities for Linear-Gaussian Policies}\label{app:identities}

This appendix proves the LQ identities used in the convergence analysis.

\subsection{Quadratic value representation}

We begin with the quadratic representation of the value and soft action-value functions for a fixed linear-Gaussian policy.
\begin{proof}[Proof of Lemma~\ref{lem:value_representation}]
Fix an admissible pair \((K,\Sigma)\). Under \(\pi_{K,\Sigma}\),
\[
u_t=-Kx_t+\xi_t,
\qquad
\xi_t\sim\cN(0,\Sigma),
\]
and hence
\[
x_{t+1}=F_Kx_t+\eta_t,
\qquad
F_K:=A-BK,
\qquad
\Cov(\eta_t)=\Omega_\Sigma:=B\Sigma B^\top+W.
\]

Since \(\rho(\sqrt{\gamma}F_K)<1\), the series
\[
P_K
=
\sum_{t=0}^{\infty}
\gamma^t(F_K^t)^\top L_KF_K^t,
\qquad
L_K:=Q+K^\top RK,
\]
converges. Shifting the summation index gives
\[
P_K=L_K+\gamma F_K^\top P_KF_K.
\]
Conversely, iterating this Lyapunov equation shows that any solution must
coincide with the preceding series. Thus \(P_K\) is its unique
positive-semidefinite solution.

We next determine the constant term of the value function. Define
\[
r_\Sigma
:=
\Tr(R\Sigma)
-\frac{\tau}{2}
\left(
m+\log\bigl((2\pi)^m\det\Sigma\bigr)
\right).
\]
The expected one-step regularized cost under \(\pi_{K,\Sigma}\) is
\[
\E\!\left[
c(x,u)+\tau\log\pi_{K,\Sigma}(u\mid x)
\,\middle|\,x
\right]
=
x^\top L_Kx+r_\Sigma.
\]

Consider the quadratic function
\[
\widetilde V(x):=x^\top P_Kx+q.
\]
Using \(\E[\eta_t]=0\) and
\(\Cov(\eta_t)=\Omega_\Sigma\), we have
\[
\E[\widetilde V(x_{t+1})\mid x_t=x]
=
x^\top F_K^\top P_KF_Kx
+
\Tr(P_K\Omega_\Sigma)
+
q.
\]
Therefore
\[
\begin{aligned}
&\E\!\left[
c(x,u)+\tau\log\pi_{K,\Sigma}(u\mid x)
+\gamma\widetilde V(x_{t+1})
\,\middle|\,x_t=x
\right]
\\
&\quad=
x^\top
\left(
L_K+\gamma F_K^\top P_KF_K
\right)x
+
r_\Sigma
+
\gamma\Tr(P_K\Omega_\Sigma)
+
\gamma q.
\end{aligned}
\]
By the Lyapunov equation, the quadratic coefficient equals \(P_K\).
Thus \(\widetilde V\) satisfies the policy Bellman equation provided
\[
(1-\gamma)q
=
r_\Sigma+\gamma\Tr(P_K\Omega_\Sigma).
\]
Since
\[
\begin{aligned}
\Tr(R\Sigma)
+\gamma\Tr(P_KB\Sigma B^\top)
&=
\Tr\!\left(
\Sigma(R+\gamma B^\top P_KB)
\right) \\
&=
\Tr(\Sigma M_K),
\end{aligned}
\]
where
\[
M_K:=R+\gamma B^\top P_KB,
\]
the required constant is
\[
q=q_{K,\Sigma}
=
\frac{1}{1-\gamma}
\left(
\Tr(\Sigma M_K)
-\frac{\tau}{2}
\left(
m+\log\bigl((2\pi)^m\det\Sigma\bigr)
\right)
+\gamma\Tr(WP_K)
\right).
\]
Iterating the policy Bellman identity over a finite horizon and using
\[
\gamma^T
\E\!\left[
1+\|x_T\|^2
\,\middle|\,x_0=x
\right]
\longrightarrow0
\]
under discounted stability shows that
\[
V^{\pi_{K,\Sigma}}(x)
=
x^\top P_Kx+q_{K,\Sigma}.
\]

We now compute the corresponding soft state-action value. By definition,
\[
Q^{\pi_{K,\Sigma}}(x,u)
=
c(x,u)
+
\gamma
\E\!\left[
V^{\pi_{K,\Sigma}}(Ax+Bu+w)
\right].
\]
Substituting the quadratic value representation gives
\[
\begin{aligned}
Q^{\pi_{K,\Sigma}}(x,u)
&=
x^\top Qx+u^\top Ru
+
\gamma(Ax+Bu)^\top P_K(Ax+Bu)
\\
&\quad
+\gamma\Tr(WP_K)
+\gamma q_{K,\Sigma}.
\end{aligned}
\]
Hence, with
\[
N_K:=\gamma B^\top P_KA,
\]
we obtain
\[
Q^{\pi_{K,\Sigma}}(x,u)
=
u^\top M_Ku
+
2u^\top N_Kx
+
\psi_{K,\Sigma}(x),
\]
where
\[
\psi_{K,\Sigma}(x)
=
x^\top
\left(
Q+\gamma A^\top P_KA
\right)x
+
\gamma\Tr(WP_K)
+
\gamma q_{K,\Sigma}.
\]

Finally, taking expectation over \(x_0\sim\cD\) yields
\[
C(K,\Sigma)
=
\Tr(\Gamma_0P_K)+q_{K,\Sigma}.
\]
\end{proof}

\subsection{Proof of the optimality of a linear-Gaussian policy}

\begin{proof}[Proof of Proposition~\ref{prop:optimizer}]
Apply the standard discrete-time Riccati theorem to the scaled system
\[
\widetilde A:=\sqrt{\gamma}A,
\qquad
\widetilde B:=\sqrt{\gamma}B.
\]
Under Assumption~\ref{ass:standard}(i)-(ii), the discounted algebraic Riccati
equation
\[
P
=
Q+\gamma A^\top PA
-
\gamma^2 A^\top PB
\left(R+\gamma B^\top PB\right)^{-1}
B^\top PA
\]
has a unique stabilizing solution \(P_\star\succeq0\).

Define
\[
M_\star
:=
R+\gamma B^\top P_\star B,
\qquad
N_\star
:=
\gamma B^\top P_\star A,
\]
and
\[
K_\star
:=
(M_\star)^{-1}N_\star,
\qquad
\Sigma_\star
:=
\frac{\tau}{2}(M_\star)^{-1}.
\]
Since \(M_\star\succeq R\succ0\), we have \(\Sigma_\star\succ0\).
Moreover, the stabilizing property of \(P_\star\) gives
\[
\rho\!\left(
\sqrt{\gamma}(A-BK_\star)
\right)<1.
\]

Set
\[
q_\star
:=
\frac{1}{1-\gamma}
\left(
\gamma\Tr(P_\star W)
+\frac{\tau}{2}\log\det M_\star
-\frac{\tau m}{2}\log(\pi\tau)
\right),
\]
and define the quadratic candidate
\[
\overline V(x)
:=
x^\top P_\star x+q_\star.
\]

Fix \(x\in\R^n\). Define the density
\[
g_x(u)
:=
\frac{\det(M_\star)^{1/2}}{(\pi\tau)^{m/2}}
\exp\left(
-\frac{1}{\tau}
(u+K_\star x)^\top
M_\star
(u+K_\star x)
\right).
\]
This is the density of
\[
\pi_\star(\cdot\mid x)
=
\cN(-K_\star x,\Sigma_\star).
\]

Let \(\rho(\cdot\mid x)\) be any probability density for which the
following expression is finite. Since
\(M_\star K_\star=N_\star\), completing the square gives
\[
u^\top M_\star u+2u^\top N_\star x
=
(u+K_\star x)^\top M_\star(u+K_\star x)
-
x^\top(N_\star)^\top(M_\star)^{-1}N_\star x.
\]
By the definition of \(g_x\),
\[
\begin{aligned}
&u^\top M_\star u
+2u^\top N_\star x
+\tau\log\rho(u\mid x)
\\
&=
-x^\top(N_\star)^\top(M_\star)^{-1}N_\star x
+\frac{\tau}{2}\log\det M_\star
-\frac{\tau m}{2}\log(\pi\tau)
+\tau\log\frac{\rho(u\mid x)}{g_x(u)}.
\end{aligned}
\]
Integrating with respect to \(\rho(\cdot\mid x)\) yields
\[
\begin{aligned}
&\int_{\R^m}
\left(
u^\top M_\star u
+2u^\top N_\star x
+\tau\log\rho(u\mid x)
\right)
\rho(du\mid x)
\\
&=
-x^\top(N_\star)^\top(M_\star)^{-1}N_\star x
+\frac{\tau}{2}\log\det M_\star
-\frac{\tau m}{2}\log(\pi\tau)
\\
&\quad
+\tau\KL\left(
\rho(\cdot\mid x)
\,\middle\|\,
\pi_\star(\cdot\mid x)
\right).
\end{aligned}
\]

On the other hand,
\[
\begin{aligned}
&c(x,u)
+\gamma\E[\overline V(Ax+Bu+w)]
\\
&=
x^\top
\left(
Q+\gamma A^\top P_\star A
\right)x
+\gamma\Tr(P_\star W)
+\gamma q_\star
\\
&\quad
+u^\top M_\star u
+2u^\top N_\star x.
\end{aligned}
\]
Combining the preceding two identities, we obtain
\[
\begin{aligned}
&\int_{\R^m}
\Bigl(
c(x,u)
+\gamma\E[\overline V(Ax+Bu+w)]
+\tau\log\rho(u\mid x)
\Bigr)
\rho(du\mid x)
\\
&=
x^\top
\left(
Q+\gamma A^\top P_\star A
-(N_\star)^\top(M_\star)^{-1}N_\star
\right)x
\\
&\quad
+\gamma\Tr(P_\star W)
+\gamma q_\star
+\frac{\tau}{2}\log\det M_\star
-\frac{\tau m}{2}\log(\pi\tau)
\\
&\quad
+\tau\KL\left(
\rho(\cdot\mid x)
\,\middle\|\,
\pi_\star(\cdot\mid x)
\right).
\end{aligned}
\]
The Riccati equation implies
\[
Q+\gamma A^\top P_\star A
-(N_\star)^\top(M_\star)^{-1}N_\star
=
P_\star,
\]
while the definition of \(q_\star\) implies
\[
\gamma\Tr(P_\star W)
+\gamma q_\star
+\frac{\tau}{2}\log\det M_\star
-\frac{\tau m}{2}\log(\pi\tau)
=
q_\star.
\]
Therefore
\[
\begin{aligned}
&\int_{\R^m}
\Bigl(
c(x,u)
+\gamma\E[\overline V(Ax+Bu+w)]
+\tau\log\rho(u\mid x)
\Bigr)
\rho(du\mid x)
\\
&=
\overline V(x)
+
\tau\KL\left(
\rho(\cdot\mid x)
\,\middle\|\,
\pi_\star(\cdot\mid x)
\right).
\end{aligned}
\tag{*}
\]
Since relative entropy is nonnegative and vanishes only when its two
arguments agree almost everywhere, taking the infimum over all action
densities gives
\[
\cT_*\overline V
=
\cT_{\pi_\star}\overline V
=
\overline V.
\]
Moreover, \(\pi_\star(\cdot\mid x)\) is the unique minimizer of the
soft Bellman expression for every \(x\).

Let \(\pi\in\Pi_{\mathrm{adm}}\). Applying \((*)\) with
\(\rho(\cdot\mid x)=\pi(\cdot\mid x)\) along a trajectory generated by
\(\pi\) gives
\[
\begin{aligned}
&\E_x^\pi
\left[
c(x_t,u_t)
+\tau\log\pi(u_t\mid x_t)
+\gamma\overline V(x_{t+1})
\,\middle|\,
x_t
\right]
\\
&=
\overline V(x_t)
+
\tau\KL\left(
\pi(\cdot\mid x_t)
\,\middle\|\,
\pi_\star(\cdot\mid x_t)
\right).
\end{aligned}
\]
Multiplying by \(\gamma^t\), taking expectations, and summing from
\(t=0\) to \(T-1\), the value terms telescope and yield
\[
\begin{aligned}
&\E_x^\pi
\left[
\sum_{t=0}^{T-1}
\gamma^t
\left(
c(x_t,u_t)
+\tau\log\pi(u_t\mid x_t)
\right)
\right]
\\
&=
\overline V(x)
-
\gamma^T\E_x^\pi[\overline V(x_T)]
\\
&\quad
+\tau
\E_x^\pi
\left[
\sum_{t=0}^{T-1}
\gamma^t
\KL\left(
\pi(\cdot\mid x_t)
\,\middle\|\,
\pi_\star(\cdot\mid x_t)
\right)
\right].
\end{aligned}
\]
By the definition of \(\Pi_{\mathrm{adm}}\) and the quadratic growth of
\(\overline V\),
\[
\gamma^T\E_x^\pi[|\overline V(x_T)|]\longrightarrow0.
\]
Letting \(T\to\infty\), we obtain
\[
\begin{aligned}
V^\pi(x)
&=
\overline V(x)
\\
&\quad
+\tau
\E_x^\pi
\left[
\sum_{t=0}^{\infty}
\gamma^t
\KL\left(
\pi(\cdot\mid x_t)
\,\middle\|\,
\pi_\star(\cdot\mid x_t)
\right)
\right]
\\
&\ge
\overline V(x).
\end{aligned}
\tag{**}
\]
For \(\pi=\pi_\star\), every relative-entropy term vanishes.
Furthermore, discounted stability implies the required transversality,
and hence
\[
V^{\pi_\star}(x)=\overline V(x).
\]
Combining this equality with \((**)\) gives~\eqref{eq:gaussian_sufficiency}. If another policy is optimal for every initial state, then the time-zero relative-entropy term vanishes for every \(x\), so its conditional density agrees almost everywhere with \(\pi_\star(\cdot\mid x)\). This proves uniqueness. Because \(\pi_\star\) is admissible and linear-Gaussian, equality~\eqref{eq:wlog_gaussian} follows.

It remains to express the result in the policy-evaluation notation.
Using \(K_\star=(M_\star)^{-1}N_\star\), the Riccati equation can be
rewritten as
\[
P_\star
=
Q+(K_\star)^\top RK_\star
+
\gamma
(A-BK_\star)^\top
P_\star
(A-BK_\star).
\]
Since \(K_\star\) is stabilizing, this discounted Lyapunov equation has
a unique positive-semidefinite solution. Therefore
\[
P_{K_\star}=P_\star.
\]
Consequently,
\[
M_{K_\star}=M_\star,
\qquad
N_{K_\star}=N_\star,
\]
and hence
\[
K_\star
=
M_{K_\star}^{-1}N_{K_\star},
\qquad
E_{K_\star}
=
M_{K_\star}K_\star-N_{K_\star}
=
0,
\]
as well as
\[
\Sigma_\star
=
\frac{\tau}{2}M_{K_\star}^{-1}.
\]
\end{proof}

\subsection{Soft Bellman residual and Gaussian log-Sobolev inequality}

\begin{lemma}[Gaussian KL formula]\label{lem:gauss_kl}
Let
\[
\nu_1=\cN(m_1,\Sigma_1),
\qquad
\nu_2=\cN(m_2,\Sigma_2).
\]
Then
\[
\KL(\nu_1\|\nu_2)
=
\frac12
\left[
(m_1-m_2)^\top\Sigma_2^{-1}(m_1-m_2)
+
\Tr(\Sigma_2^{-1}\Sigma_1)
-m
-\log\det(\Sigma_2^{-1}\Sigma_1)
\right].
\]
\end{lemma}

\begin{proof}
Using the Gaussian density formula,
\[
\log\frac{\dd\nu_1}{\dd\nu_2}(u)
=
\frac12\log\frac{\det\Sigma_2}{\det\Sigma_1}
+
\frac12(u-m_2)^\top\Sigma_2^{-1}(u-m_2)
-
\frac12(u-m_1)^\top\Sigma_1^{-1}(u-m_1).
\]
Taking expectation under \(\nu_1\), and using
\[
\E_{\nu_1}\!\left[
(u-m_1)^\top\Sigma_1^{-1}(u-m_1)
\right]
=m
\]
and
\[
\E_{\nu_1}\!\left[
(u-m_2)^\top\Sigma_2^{-1}(u-m_2)
\right]
=
\Tr(\Sigma_2^{-1}\Sigma_1)
+
(m_1-m_2)^\top\Sigma_2^{-1}(m_1-m_2),
\]
gives the result.
\end{proof}

\begin{proof}[Proof of Lemma~\ref{lem:bellman_lsi}]
\proofpart{Gibbs policy and Bellman residual}
By Lemma~\ref{lem:value_representation}, we have
\[
V_{K,\Sigma}(x)=x^\top P_Kx+q_{K,\Sigma}.
\]
For the fixed policy \((K,\Sigma)\), the state-action function
\[
Q_{K,\Sigma}(x,u):=c(x,u)+\gamma\E\!\left[V_{K,\Sigma}(Ax+Bu+w)\right]
\]
is quadratic in \(u\). More precisely, expanding \(c(x,u)=x^\top Qx+u^\top Ru\) and the term \(\gamma\E[(Ax+Bu+w)^\top P_K(Ax+Bu+w)]\), and collecting the \(u\)-quadratic and \(u\)-linear coefficients, gives
\[
Q_{K,\Sigma}(x,u)=u^\top M_Ku+2u^\top N_Kx+\psi_{K,\Sigma}(x),
\]
where \(\psi_{K,\Sigma}\) is independent of \(u\), \(M_K=R+\gamma B^\top P_KB\), and \(N_K=\gamma B^\top P_KA\).

It follows that
\[
\exp\!\left(-\frac{1}{\tau}Q_{K,\Sigma}(x,u)\right)
\propto
\exp\!\left(-\frac{1}{\tau}\left(u^\top M_Ku+2u^\top N_Kx\right)\right).
\]
Completing the square yields
\[
p_K(\cdot\mid x)=\cN\!\left(-M_K^{-1}N_Kx,\Xi_K\right),
\qquad
\Xi_K:=\frac{\tau}{2}M_K^{-1}.
\]

For any density \(\rho\), the Gibbs variational identity gives
\[
\int\left(Q_{K,\Sigma}(x,u)+\tau\log\rho(u)\right)\rho(\dd u)
=
(\cT_*V_{K,\Sigma})(x)
+
\tau\KL\!\left(\rho\,\middle\|\,p_K(\cdot\mid x)\right).
\]
Indeed, if \(Z_{K,\Sigma}(x):=\int\exp(-Q_{K,\Sigma}(x,u)/\tau)\,\dd u\), then
\[
(\cT_*V_{K,\Sigma})(x)=-\tau\log Z_{K,\Sigma}(x),
\]
and the preceding identity follows from
\[
\tau\KL\!\left(\rho\,\middle\|\,p_K(\cdot\mid x)\right)
=
\int\left(Q_{K,\Sigma}(x,u)+\tau\log\rho(u)\right)\rho(\dd u)
+
\tau\log Z_{K,\Sigma}(x).
\]
Taking \(\rho=\pi_{K,\Sigma}(\cdot\mid x)\) and using the Bellman identity
\[
V_{K,\Sigma}=\cT_{\pi_{K,\Sigma}}V_{K,\Sigma},
\]
we obtain
\[
R_{K,\Sigma}(x)
:=
V_{K,\Sigma}(x)-(\cT_*V_{K,\Sigma})(x)
=
\tau\KL\!\left(\pi_{K,\Sigma}(\cdot\mid x)\,\middle\|\,p_K(\cdot\mid x)\right).
\]

Applying the Gaussian KL formula in Lemma~\ref{lem:gauss_kl} to
\[
\pi_{K,\Sigma}(\cdot\mid x)=\cN(-Kx,\Sigma)
\quad\text{and}\quad
p_K(\cdot\mid x)=\cN(-M_K^{-1}N_Kx,\Xi_K),
\]
we get
\[
R_{K,\Sigma}(x)
=
x^\top (K-M_K^{-1}N_K)^\top M_K(K-M_K^{-1}N_K)x
+
\frac{\tau}{2}
\left[
\Tr(\Xi_K^{-1}\Sigma)-m-\log\det(\Xi_K^{-1}\Sigma)
\right].
\]
Since \(E_K=M_KK-N_K\),
\[
(K-M_K^{-1}N_K)^\top M_K(K-M_K^{-1}N_K)=E_K^\top M_K^{-1}E_K.
\]
This gives the explicit formula~\eqref{eq:bellman_kl}.

\proofpart{Gaussian log-Sobolev inequality}
The Gibbs conditional action law
\[
p_K(\cdot\mid x)=\cN(-M_K^{-1}N_Kx,\Xi_K)
\]
has covariance
\[
\Xi_K=\frac{\tau}{2}M_K^{-1}.
\]
A nondegenerate Gaussian measure \(\mathsf G_\Xi\) with covariance \(\Xi\) satisfies
\[
\mathcal I(\nu\|\mathsf G_\Xi)
\ge
2\lambda_{\min}(\Xi^{-1})\KL(\nu\|\mathsf G_\Xi),
\]
and the coefficient \(\lambda_{\min}(\Xi^{-1})\) is sharp; see, for example,
\citet{bakry2014analysis}. Hence, for the current Gibbs law,
\[
    \alpha_K
    :=
    \lambda_{\min}(\Xi_K^{-1})
    =
    \frac{2\lambda_{\min}(M_K)}{\tau}.
\]
Equivalently, if the LSI is written as
\[
    \KL(\nu\|p_K)
    \le
    \frac{C_{\mathrm{LSI}}(p_K)}{2}\mathcal I(\nu\|p_K),
\]
then its sharp constant is
\[
    C_{\mathrm{LSI}}(p_K)
    =
    \frac{1}{\alpha_K}
    =
    \frac{\tau}{2\lambda_{\min}(M_K)}.
\]
Since
\[
M_K=R+\gamma B^\top P_KB\succeq R\succeq\lambda_R I,
\]
we also have the uniform lower bound
\[
    \alpha_K
    \ge
    \underline\alpha
    :=
    \frac{2\lambda_R}{\tau}.
\]
Therefore
\[
\mathcal I(\nu\|p_K(\cdot\mid x))
\ge
2\alpha_K\KL(\nu\|p_K(\cdot\mid x))
\ge
2\underline\alpha\KL(\nu\|p_K(\cdot\mid x)).
\]
Translations do not change either Gaussian coefficient, so the uniform estimate is independent of \(x\) and of the admissible gain \(K\).
\end{proof}

\subsection{Performance-difference identity and gradients}

\begin{proof}[Proof of Lemma~\ref{lem:perf}]
Let
\[
\Delta K:=K'-K.
\]
Let \((x_t,u_t)_{t\ge0}\) be the trajectory generated by the policy \((K',\Sigma')\). Since \(V_{K,\Sigma}\) has quadratic growth and \((K',\Sigma')\) is admissible, Lemma~\ref{lem:v2_resolvent} implies
\[
\gamma^T\E^{K',\Sigma'}\!\left[V_{K,\Sigma}(x_T)\right]\to0
\qquad
\text{as }T\to\infty.
\]
Using \(C(K,\Sigma)=\E_{x_0\sim \cD}[V_{K,\Sigma}(x_0)]\), the usual telescoping argument gives
\begin{align*}
C(K',\Sigma')-C(K,\Sigma)
=
\E^{K',\Sigma'}\sum_{t=0}^{\infty}\gamma^t\Big(
&c(x_t,u_t)+\tau\log\pi_{K',\Sigma'}(u_t\mid x_t) \\
&+\gamma V_{K,\Sigma}(x_{t+1})-V_{K,\Sigma}(x_t)
\Big).
\end{align*}

We next compute the conditional one-step contribution. Conditioning on \(x_t=x\), let
\[
u\sim\cN(-K'x,\Sigma'),
\qquad
x_+:=Ax+Bu+w.
\]
Lemma~\ref{lem:value_representation} gives
\[
V_{K,\Sigma}(x)=x^\top P_Kx+q_{K,\Sigma},
\qquad
M_K=R+\gamma B^\top P_KB,
\qquad
    N_K=\gamma B^\top P_KA.
\]
A direct computation then gives
\begin{align*}
&\E\!\left[
c(x,u)+\tau\log\pi_{K',\Sigma'}(u\mid x)+\gamma V_{K,\Sigma}(x_+)-V_{K,\Sigma}(x)
\,\middle|\,x_t=x
\right] \\
&\qquad =
x^\top\!\left(K'^\top M_KK'-2K'^\top N_K-K^\top M_KK+2K^\top N_K\right)x \\
&\qquad\quad
+\Tr(M_K\Sigma')-\frac{\tau}{2}\log\det\Sigma'
-\Tr(M_K\Sigma)+\frac{\tau}{2}\log\det\Sigma.
\end{align*}
As a quadratic form in \(x\), the matrix coefficient agrees with
\[
\Delta K^\top M_K\Delta K+2\Delta K^\top E_K,
\qquad
E_K=M_KK-N_K.
\]
Indeed, the two matrix representatives can differ only by a skew-symmetric term, which vanishes inside \(x^\top(\cdot)x\). The constant part is
\[
(1-\gamma)\bigl(f_K(\Sigma)-f_K(\Sigma')\bigr).
\]

Substituting the one-step expansion into the telescoping identity and summing the quadratic term along the trajectory of \((K',\Sigma')\), we obtain
\begin{align*}
&\E^{K',\Sigma'}\sum_{t=0}^{\infty}\gamma^t
x_t^\top\left(\Delta K^\top M_K\Delta K+2\Delta K^\top E_K\right)x_t \\
&\qquad =
\Tr\!\left(
\Scorr_{K',\Sigma'}
\left(\Delta K^\top M_K\Delta K+2\Delta K^\top E_K\right)
\right).
\end{align*}
The constant term contributes
\[
\sum_{t=0}^{\infty}\gamma^t(1-\gamma)\bigl(f_K(\Sigma)-f_K(\Sigma')\bigr)
=
f_K(\Sigma)-f_K(\Sigma').
\]
Therefore
\[
C(K',\Sigma')-C(K,\Sigma)
=
\Tr\!\left(
\Scorr_{K',\Sigma'}
\left(\Delta K^\top M_K\Delta K+2\Delta K^\top E_K\right)
\right)
+
f_K(\Sigma)-f_K(\Sigma').
\]
This proves~\eqref{eq:perf_diff}.

It remains to identify the gradients. We differentiate~\eqref{eq:perf_diff} at
\[
(K',\Sigma')=(K,\Sigma).
\]
First fix \(\Sigma'=\Sigma\) and set
\[
K'=K+\varepsilon H.
\]
Then \(\Delta K=\varepsilon H\). Since the bracketed term in~\eqref{eq:perf_diff} vanishes at \(\varepsilon=0\), the first variation of \(\Scorr_{K+\varepsilon H,\Sigma}\) does not contribute. Thus
\[
C(K+\varepsilon H,\Sigma)-C(K,\Sigma)
=
2\varepsilon\Tr\!\left(\Scorr_{K,\Sigma}H^\top E_K\right)
+
o(\varepsilon).
\]
With respect to the Frobenius inner product, this gives
\[
\nabla_K C(K,\Sigma)=2E_K\Scorr_{K,\Sigma}.
\]

Next fix \(K'=K\) and set
\[
\Sigma'=\Sigma+\varepsilon U,
\qquad
U=U^\top.
\]
The quadratic term in~\eqref{eq:perf_diff} vanishes identically, so
\[
C(K,\Sigma+\varepsilon U)-C(K,\Sigma)
=
f_K(\Sigma)-f_K(\Sigma+\varepsilon U).
\]
Using
\[
f_K(\Sigma)
=
\frac{\tau}{2(1-\gamma)}\log\det\Sigma
-
\frac{1}{1-\gamma}\Tr(M_K\Sigma),
\]
we obtain
\[
C(K,\Sigma+\varepsilon U)-C(K,\Sigma)
=
\frac{\varepsilon}{1-\gamma}
\Tr\!\left(
\left(M_K-\frac{\tau}{2}\Sigma^{-1}\right)U
\right)
+
o(\varepsilon).
\]
Hence
\[
\nabla_\Sigma C(K,\Sigma)
=
\frac{1}{1-\gamma}
\left(M_K-\frac{\tau}{2}\Sigma^{-1}\right).
\]
This proves~\eqref{eq:gradients}.
\end{proof}

\section{Closure on Linear-Gaussian Policies and Value Dissipation}\label{app:flow}

This appendix verifies the policy first variation along compactly supported action transports at a linear-Gaussian policy and proves the closure and dissipation identities.

\subsection{Policy-space metric and Wasserstein gradient}

\begin{lemma}[First variation at a linear-Gaussian policy]
\label{lem:conditional_first_variation}
Fix \((K,\Sigma)\in\cA_{\mathrm{adm}}\) and write
\(\pi=\pi_{K,\Sigma}\). Let
\(\varphi\in C_c^\infty(\R^n\times\R^m)\), set
\(v=\nabla_u\varphi\), and define
\[
    T_{\varepsilon,x}(u)
    :=u+\varepsilon v(x,u),
    \qquad
    \pi^\varepsilon(\cdot\mid x)
    :=(T_{\varepsilon,x})_\#\pi(\cdot\mid x).
\]
There exists \(\varepsilon_0>0\) such that, for every
\(\lvert\varepsilon\rvert<\varepsilon_0\), the map
\(T_{\varepsilon,x}\) is a smooth bijection of \(\R^m\) for every \(x\),
the policy \(\pi^\varepsilon\) has a positive smooth density and belongs to
\(\Pi_{\mathrm{adm}}\), and the map
\(\varepsilon\mapsto C(\pi^\varepsilon)\) is differentiable at zero. Its
derivative is
\begin{equation}
\label{eq:objective_first_variation_expansion}
\begin{aligned}
\left.\frac{\dd}{\dd\varepsilon}C(\pi^\varepsilon)\right|_{\varepsilon=0}
&=
\frac{1}{1-\gamma}
\int_{\R^n}d_{\mathcal D}^{\pi}(\dd x)
\int_{\R^m}
\bigl(Q^\pi(x,u)+\tau\log\pi(u\mid x)\bigr)
\xi_v(u\mid x)\,\dd u                                      \\
&=
\frac{1}{1-\gamma}
\int_{\R^n}d_{\mathcal D}^{\pi}(\dd x)
\int_{\R^m}
\left\langle
\nabla_u\bigl(Q^\pi(x,u)+\tau\log\pi(u\mid x)\bigr),
 v(x,u)
\right\rangle
\pi(\dd u\mid x),
\end{aligned}
\end{equation}
where
\[
    \xi_v(u\mid x)
    :=
    \left.\frac{\dd}{\dd\varepsilon}
    \pi^\varepsilon(u\mid x)\right|_{\varepsilon=0}
    =
    -\nabla_u\cdot\bigl(\pi(u\mid x)v(x,u)\bigr).
\]
In particular, the differentiation under the discounted infinite sum, the
Fubini interchange, and the integration by parts in
\eqref{eq:objective_first_variation_expansion} are valid under the standing
assumptions.
\end{lemma}

\begin{proof}
\proofpart{The transport path}
Because \(v\) and \(\nabla_uv\) are bounded and vanish outside a compact
subset of \(\R^n\times\R^m\), choose \(\varepsilon_0>0\) so that
\[
    \lvert\varepsilon\rvert\|\nabla_uv\|_\infty<\frac12
    \qquad
    \text{for }\lvert\varepsilon\rvert<\varepsilon_0.
\]
For every fixed \(x\),
\[
\|T_{\varepsilon,x}(u)-T_{\varepsilon,x}(a)\|
\ge
\bigl(1-\lvert\varepsilon\rvert\|\nabla_uv\|_\infty\bigr)
\|u-a\|.
\]
Thus \(T_{\varepsilon,x}\) is injective and has an invertible Jacobian. It
agrees with the identity outside a compact set in the action variable, and is
therefore proper. Its image is open by the inverse function theorem and closed
by properness. Since \(\R^m\) is connected, the image is all of \(\R^m\), and
the inverse is smooth. The change-of-variables formula gives
\begin{equation}
\label{eq:transport_density_identity}
\pi^\varepsilon(T_{\varepsilon,x}(u)\mid x)
\det\bigl(I_m+\varepsilon\nabla_uv(x,u)\bigr)
=
\pi(u\mid x).
\end{equation}
Hence \(\pi^\varepsilon\) has a positive smooth density. Differentiating at
zero yields
\begin{equation}
\label{eq:transport_density_derivative}
\xi_v=-\nabla_u\cdot(\pi v).
\end{equation}
Writing \(s_v:=\xi_v/\pi\),
\begin{equation}
\label{eq:transport_score}
s_v(x,u)
=
-\nabla_u\cdot v(x,u)
-
\left\langle v(x,u),\nabla_u\log\pi(u\mid x)\right\rangle.
\end{equation}
Since
\(\nabla_u\log\pi(u\mid x)=-\Sigma^{-1}(u+Kx)\) and \(v\) has compact
support in \((x,u)\), the score \(s_v\) is bounded and compactly supported.
Moreover,
\(\int_{\R^m}\xi_v(u\mid x)\,\dd u=0\) for every \(x\).

\proofpart{Uniform discounted moment bound}
Set \(F=A-BK\). Admissibility permits a number
\[
    b\in\bigl(\max\{1,\rho(F)\},\gamma^{-1/2}\bigr)
\]
and a constant \(c_F<\infty\) such that
\(\|F^j\|\le c_Fb^j\) for every \(j\ge0\). A draw from
\(\pi^\varepsilon(\cdot\mid x)\) can be represented as
\[
    u=-Kx+\xi+\varepsilon v(x,-Kx+\xi),
    \qquad
    \xi\sim\cN(0,\Sigma).
\]
The state recursion is therefore
\[
    x_{t+1}
    =Fx_t+B\xi_t+w_t
    +\varepsilon Bv(x_t,-Kx_t+\xi_t).
\]
The last term is uniformly bounded, while \(B\xi_t+w_t\) has a finite
second moment. Iterating the recursion and applying Minkowski's inequality
shows that, for every deterministic initial state \(x\),
\[
\bigl(\E_x^{\pi^\varepsilon}\|x_t\|^2\bigr)^{1/2}
\le
C b^t(1+\|x\|),
\]
uniformly for sufficiently small \(\varepsilon\). The action representation
gives the same type of bound for \(u_t\). Hence, for some
\(C<\infty\) and \(q\in(0,1)\),
\begin{equation}
\label{eq:first_variation_moment_bound}
\gamma^t\E_x^{\pi^\varepsilon}
\left[1+\|x_t\|^2+\|u_t\|^2\right]
\le
Cq^t(1+\|x\|^2),
\qquad t\ge0,
\end{equation}
uniformly for small \(\varepsilon\). One may take any
\(q\in(\gamma b^2,1)\) after increasing \(C\). This proves the discounted
transversality condition and shows that \(\pi^\varepsilon\in\Pi_{\mathrm{adm}}\).
After integration over \(x_0\sim\mathcal D\), the same estimate holds without
the factor \(1+\|x\|^2\) on the right.

The change-of-variables identity also gives, uniformly for small
\(\varepsilon\),
\begin{equation}
\label{eq:perturbed_stage_growth}
\left|c(x,u)+\tau\log\pi^\varepsilon(u\mid x)\right|
\le
C\bigl(1+\|x\|^2+\|u\|^2\bigr).
\end{equation}
Indeed, the Gaussian log density is quadratic, the Jacobian determinant is
bounded above and away from zero, and the transport differs from the identity
only on a fixed compact set.

\proofpart{Differentiation of the discounted series}
On a trajectory generated under the reference policy \(\pi\), define
\[
    r_\varepsilon(x,u)
    :=\frac{\pi^\varepsilon(u\mid x)}{\pi(u\mid x)},
    \qquad
    L_t^\varepsilon
    :=\prod_{j=0}^t r_\varepsilon(x_j,u_j).
\]
The initial-state law and transition kernel are the same under the two
policies, so
\begin{equation}
\label{eq:objective_reference_policy}
C(\pi^\varepsilon)
=
\sum_{t=0}^\infty\gamma^t
\E_\pi\!\left[
\bigl(c(x_t,u_t)+\tau\log\pi^\varepsilon(u_t\mid x_t)\bigr)
L_t^\varepsilon
\right].
\end{equation}
By~\eqref{eq:transport_density_identity} and compact support of \(v\), there
is \(c<\infty\) such that, uniformly in \((x,u)\),
\[
    e^{-c\lvert\varepsilon\rvert}
    \le r_\varepsilon(x,u)
    \le e^{c\lvert\varepsilon\rvert},
    \qquad
    \left|\partial_\varepsilon\log r_\varepsilon(x,u)\right|
    \le c,
\]
for sufficiently small \(\varepsilon\). The same argument bounds
\(\lvert\partial_\varepsilon\log\pi^\varepsilon\rvert\) on the set where it
is nonzero. By~\eqref{eq:perturbed_stage_growth}, the derivative of the
\(t\)-th integrand in~\eqref{eq:objective_reference_policy} is bounded in
absolute value by
\[
C(t+1)e^{c\lvert\varepsilon\rvert(t+1)}
\bigl(1+\|x_t\|^2+\|u_t\|^2\bigr).
\]
Choose \(\varepsilon_0\) smaller if needed so that
\(qe^{c\varepsilon_0}<1\), where \(q\) is from
\eqref{eq:first_variation_moment_bound}. The discounted expectations of these
bounds form a summable series. Dominated convergence therefore justifies
differentiation under the expectation and the infinite sum.

At \(\varepsilon=0\),
\[
\left.\partial_\varepsilon r_\varepsilon(x,u)\right|_{0}
=
\left.\partial_\varepsilon\log\pi^\varepsilon(u\mid x)\right|_{0}
=
s_v(x,u).
\]
Consequently,
\begin{align}
\left.\frac{\dd}{\dd\varepsilon}C(\pi^\varepsilon)\right|_{0}
&=
\sum_{t=0}^\infty\gamma^t\E_\pi\!\left[
\bigl(c(x_t,u_t)+\tau\log\pi(u_t\mid x_t)\bigr)
\sum_{j=0}^t s_v(x_j,u_j)
\right]
\notag\\
&\quad
+
\tau\sum_{t=0}^\infty\gamma^t\E_\pi[s_v(x_t,u_t)].
\label{eq:trajectory_first_variation}
\end{align}
The second sum vanishes because
\[
\E_\pi[s_v(x_t,u_t)\mid x_t]
=
\int_{\R^m}\xi_v(u\mid x_t)\,\dd u
=0.
\]

\proofpart{Fubini interchange}
The score \(s_v\) is bounded. By
\eqref{eq:first_variation_moment_bound},
\[
\sum_{t=0}^\infty\sum_{j=0}^t
\gamma^t\E_\pi\!\left[
|s_v(x_j,u_j)|
\left|c(x_t,u_t)+\tau\log\pi(u_t\mid x_t)\right|
\right]
\le
C\sum_{t=0}^\infty(t+1)q^t
<\infty.
\]
Fubini's theorem therefore permits exchange of the two sums in
\eqref{eq:trajectory_first_variation}. Writing \(t=j+r\) and conditioning on
\((x_j,u_j)\), the conditional expected cost from time \(j\) onward is
\(Q^\pi(x_j,u_j)+\tau\log\pi(u_j\mid x_j)\). Hence
\[
\begin{aligned}
\left.\frac{\dd}{\dd\varepsilon}C(\pi^\varepsilon)\right|_{0}
&=
\sum_{j=0}^\infty\gamma^j\E_\pi\!\left[
 s_v(x_j,u_j)
 \bigl(Q^\pi(x_j,u_j)+\tau\log\pi(u_j\mid x_j)\bigr)
\right]\\
&=
\frac{1}{1-\gamma}
\int_{\R^n}d_{\mathcal D}^{\pi}(\dd x)
\int_{\R^m}
\bigl(Q^\pi(x,u)+\tau\log\pi(u\mid x)\bigr)
\xi_v(u\mid x)\,\dd u.
\end{aligned}
\]
This is the first equality in
\eqref{eq:objective_first_variation_expansion}.

\proofpart{Integration by parts and completion of the tangent space}
For a linear-Gaussian policy,
\(Q^\pi(x,u)+\tau\log\pi(u\mid x)\) is a smooth quadratic function of
\((x,u)\). Since \(v\) has compact support in the action variable,
\eqref{eq:transport_density_derivative} gives, with no boundary term,
\[
\begin{aligned}
&\int_{\R^m}
\bigl(Q^\pi(x,u)+\tau\log\pi(u\mid x)\bigr)
\xi_v(u\mid x)\,\dd u\\
&\qquad=
\int_{\R^m}
\left\langle
\nabla_u\bigl(Q^\pi(x,u)+\tau\log\pi(u\mid x)\bigr),
 v(x,u)
\right\rangle
\pi(\dd u\mid x).
\end{aligned}
\]
The right-hand side is absolutely integrable because the action gradient is
affine and the discounted state and action second moments are finite. This
proves the second equality in
\eqref{eq:objective_first_variation_expansion}.

Finally, let
\[
    g_\pi(x,u)
    :=\nabla_u\bigl(Q^\pi(x,u)+\tau\log\pi(u\mid x)\bigr).
\]
It is the action gradient of a quadratic function \(h_\pi\). Let
\(\chi_R(x)\) and \(\eta_L(u)\) be smooth cutoffs that equal one on balls of
radii \(R\) and \(L\), and define
\[
    \varphi_{R,L}(x,u)
    :=\chi_R(x)\eta_L(u)h_\pi(x,u).
\]
Choose the action cutoff so that \(\|\nabla_u\eta_L\|_\infty\le C/L\).
For fixed \(R\), on the annulus where \(\nabla_u\eta_L\ne0\), the
quadratic growth of \(h_\pi\) gives
\[
    \lvert h_\pi(x,u)\nabla_u\eta_L(u)\rvert
    \le C_R(1+\|u\|).
\]
The remaining cutoff error is bounded by a constant times
\(1+\|x\|+\|u\|\). First letting \(L\to\infty\) for fixed \(R\), and
then letting \(R\to\infty\), the finite second moments and dominated
convergence give
\[
    \nabla_u\varphi_{R,L}\longrightarrow g_\pi
\]
in the metric norm. Thus the affine WPG velocity lies in the completed
tangent space, and compactly supported test velocities are sufficient to
identify the metric gradient.
\end{proof}

\begin{remark}[Linear-Gaussian parameter directions]
Let \(s\mapsto(K_s,\Sigma_s)\) be a \(C^1\) linear-Gaussian parameter curve
through \((K,\Sigma)\), with derivative \((\dot K,\dot\Sigma)\). If
\(L_{\dot\Sigma}\) is the unique symmetric solution of
\[
    L_{\dot\Sigma}\Sigma+\Sigma L_{\dot\Sigma}=\dot\Sigma,
\]
then the action-gradient velocity
\[
    v_{\dot K,\dot\Sigma}(x,u)
    =-\dot Kx+L_{\dot\Sigma}(u+Kx)
\]
satisfies
\[
    \left.\frac{\dd}{\dd s}\pi_{K_s,\Sigma_s}(u\mid x)\right|_{s=0}
    =
    -\nabla_u\cdot\bigl(
    \pi_{K,\Sigma}(u\mid x)v_{\dot K,\dot\Sigma}(x,u)
    \bigr).
\]
The cutoff argument above places this affine velocity in the completed tangent
space, so the policy-space first-variation formula applies directly to every
linear-Gaussian parameter direction.
\end{remark}

\subsection{Proof of closure on the linear-Gaussian class}

\begin{proof}[Proof of Proposition~\ref{prop:gaussian_reduction}]
Fix \((K,\Sigma)\in\cA_{\mathrm{adm}}\), and abbreviate
\[
    \pi:=\pi_{K,\Sigma},
    \qquad
    M:=M_K,
    \qquad
    N:=N_K,
    \qquad
    E:=E_K=MK-N.
\]
For fixed \(x\), set
\[
    z:=u+Kx.
\]
Then \(\pi(\cdot\mid x)=\cN(-Kx,\Sigma)\) and
\[
    \nabla_u\log\pi(u\mid x)=-\Sigma^{-1}z.
\]
By~\eqref{eq:Qquad_reduction},
\[
    \nabla_uQ^\pi(x,u)=2Mu+2Nx.
\]
Since \(u=z-Kx\), the WPG transport field becomes
\[
\begin{aligned}
\nabla_u\bigl(Q^\pi(x,u)+\tau\log\pi(u\mid x)\bigr)
&=
2M(z-Kx)+2Nx-\tau\Sigma^{-1}z\\
&=
\bigl(2M-\tau\Sigma^{-1}\bigr)z-2Ex.
\end{aligned}
\]
Let
\[
    H:=2M-\tau\Sigma^{-1}.
\]
Using \(\nabla_u\pi=-\pi\Sigma^{-1}z\), the right-hand side of WPGF is
\begin{equation}\label{eq:wpg_rhs_gaussian_tangent}
\begin{aligned}
&\nabla_u\cdot\left[
\pi(u\mid x)\bigl(Hz-2Ex\bigr)
\right]\\
&\qquad=
\pi(u\mid x)
\left[
\Tr(H)
-z^\top H\Sigma^{-1}z
+2z^\top\Sigma^{-1}Ex
\right].
\end{aligned}
\end{equation}

We next compute a general tangent vector to the linear-Gaussian family. Differentiating the Gaussian density with respect to \((K,\Sigma)\) in a direction \((\dot K,\dot\Sigma)\) gives
\begin{equation}\label{eq:gaussian_tangent_general}
\begin{aligned}
D\pi_{K,\Sigma}[\dot K,\dot\Sigma](u\mid x)
=
\pi(u\mid x)
\bigg[
&-z^\top\Sigma^{-1}\dot Kx
-\frac12\Tr(\Sigma^{-1}\dot\Sigma)\\
&+\frac12z^\top\Sigma^{-1}\dot\Sigma\Sigma^{-1}z
\bigg].
\end{aligned}
\end{equation}
Choose
\[
    \dot K=-2E,
    \qquad
    \dot\Sigma=-2M\Sigma-2\Sigma M+2\tau I.
\]
The linear term in~\eqref{eq:gaussian_tangent_general} is then
\[
    -z^\top\Sigma^{-1}\dot Kx
    =
    2z^\top\Sigma^{-1}Ex.
\]
For the constant term, cyclicity of the trace gives
\[
\begin{aligned}
-\frac12\Tr(\Sigma^{-1}\dot\Sigma)
&=
\Tr(2M-\tau\Sigma^{-1})
=
\Tr(H).
\end{aligned}
\]
Finally, since \(M\) and \(\Sigma\) are symmetric,
\[
\begin{aligned}
\frac12z^\top\Sigma^{-1}\dot\Sigma\Sigma^{-1}z
&=
-z^\top\bigl(2M-\tau\Sigma^{-1}\bigr)\Sigma^{-1}z\\
&=
-z^\top H\Sigma^{-1}z.
\end{aligned}
\]
Substitution into~\eqref{eq:gaussian_tangent_general} yields exactly~\eqref{eq:wpg_rhs_gaussian_tangent}. This proves the tangent identity~\eqref{eq:tangent_identity} and the parameter velocity~\eqref{eq:KSigma_vector_field}.

The velocity is unique because the derivative of the conditional mean is \(-\dot Kx\) and the derivative of the conditional covariance is \(\dot\Sigma\); these first two moments uniquely determine the parameter velocity for every \(x\). Therefore, if \((K_t,\Sigma_t)\) satisfies~\eqref{eq:KSigma_ode}, the chain rule and~\eqref{eq:tangent_identity} show that \(\pi_t=\pi_{K_t,\Sigma_t}\) satisfies WPGF on every interval on which the parameter path is admissible. This proves exact closure on the linear-Gaussian policy class.
\end{proof}

\subsection{The closed-loop Markov resolvent on
\texorpdfstring{\(\Vtwo\)}{V2}}

\begin{lemma}[Value resolvent on \(\Vtwo\)]\label{lem:v2_resolvent}
Let \((K,\Sigma)\) be admissible. Then
\(\Pop_{K,\Sigma}\colon\Vtwo\to\Vtwo\), and there exist
\(c_{K,\Sigma}<\infty\) and \(\bar q_{K,\Sigma}\in(0,1)\) such that
\[
\|\gamma^j\Pop_{K,\Sigma}^jv\|_{\Vtwo}
\le
c_{K,\Sigma}\bar q_{K,\Sigma}^j\|v\|_{\Vtwo},
\qquad j\ge0.
\]
Consequently,
\[
(I-\gamma\Pop_{K,\Sigma})^{-1}
=
\sum_{j=0}^{\infty}\gamma^j\Pop_{K,\Sigma}^j
\]
as a bounded operator on \(\Vtwo\), and the inverse is positivity preserving.
\end{lemma}

\begin{proof}
Write \(F=F_K\), \(\Omega=\Omega_\Sigma\), and
\(\Pop=\Pop_{K,\Sigma}\). Since
\(\rho(\sqrt{\gamma}F)<1\), there exist \(c_F<\infty\) and
\(q_F\in(0,1)\) such that
\[
\gamma^j\|F^j\|^2\le c_Fq_F^j,
\qquad j\ge0.
\]
For the closed-loop process,
\[
x_j
=
F^jx+\sum_{\ell=0}^{j-1}F^{j-1-\ell}\eta_\ell,
\qquad
\E[\eta_\ell\eta_\ell^\top]=\Omega,
\]
and hence
\[
\E[\|x_j\|^2\mid x_0=x]
=
\|F^jx\|^2
+
\sum_{r=0}^{j-1}
\Tr\!\left(F^r\Omega(F^r)^\top\right).
\]
Choose
\[
\bar q_{K,\Sigma}
\in
\bigl(\max\{\gamma,q_F\},1\bigr).
\]
Using
\(\Tr(F^r\Omega(F^r)^\top)\le\Tr(\Omega)\|F^r\|^2\),
we obtain
\[
\gamma^j\E[\|x_j\|^2\mid x_0=x]
\le
c_{K,\Sigma}\bar q_{K,\Sigma}^j(1+\|x\|^2).
\]
Therefore, for \(v\in\Vtwo\),
\[
\gamma^j|\Pop^jv(x)|
\le
\|v\|_{\Vtwo}
\gamma^j\bigl(1+\E[\|x_j\|^2\mid x_0=x]\bigr),
\]
which gives
\[
\|\gamma^j\Pop^jv\|_{\Vtwo}
\le
c_{K,\Sigma}\bar q_{K,\Sigma}^j\|v\|_{\Vtwo}.
\]

Thus the Neumann series converges absolutely in operator norm. Its partial
sums satisfy
\[
(I-\gamma\Pop)
\sum_{j=0}^{T}\gamma^j\Pop^j
=
I-\gamma^{T+1}\Pop^{T+1},
\]
and the remainder converges to zero in operator norm. Hence the limit is
\((I-\gamma\Pop)^{-1}\). Since every \(\Pop^j\) is positivity preserving,
so is the inverse.
\end{proof}

\subsection{Dissipation identity for the value function}

\begin{proof}[Proof of Lemma~\ref{lem:resolvent}]
Let \(V_t:=V_{K_t,\Sigma_t}\) and
\(\Pop_t:=\Pop_{\pi_t}=\Pop_{K_t,\Sigma_t}\), and define
\[
p_t(\cdot\mid x):=\cN(-M_t^{-1}N_tx,\Xi_t),\qquad \Xi_t:=\frac{\tau}{2}M_t^{-1}.
\]
Equivalently, \(p_t(\cdot\mid x)\) is the Gibbs density associated with the function \(u\mapsto Q^{\pi_t}(x,u)\):
\[
p_t(u\mid x)=Z_t(x)^{-1}\exp\!\left(-\frac{1}{\tau}Q^{\pi_t}(x,u)\right).
\]
By Lemma~\ref{lem:value_representation},
\(V_t(x)=x^\top P_{K_t}x+q_{K_t,\Sigma_t}\), so all quantities below
belong to \(\Vtwo\). For fixed \(t\) and \(x\), introduce the action functional
\[
\mathcal J_t(\nu;x):=\int\Bigl(Q^{\pi_t}(x,u)+\tau\log\nu(u)\Bigr)\nu(u)\,\dd u.
\]
Then the soft Bellman identity reads
\[
V_t(x)=\mathcal J_t(\pi_t(\cdot\mid x);x).
\]
Differentiating this identity in \(t\) gives two contributions:
\[
\dot V_t(x)=\partial_t\mathcal J_t(\pi_t(\cdot\mid x);x)+D_\nu\mathcal J_t(\pi_t(\cdot\mid x);x)[\partial_t\pi_t(\cdot\mid x)].
\]

The first contribution comes only from the variation of \(Q^{\pi_t}\). Since
\[
\partial_t Q^{\pi_t}(x,u)=\gamma\E[\dot V_t(Ax+Bu+w)\mid x,u],
\]
we have
\[
\partial_t\mathcal J_t(\pi_t(\cdot\mid x);x)
=
\int \partial_tQ^{\pi_t}(x,u)\,\pi_t(\dd u\mid x)
=
\gamma\Pop_t\dot V_t(x).
\]

It remains to compute the action derivative. Since \(p_t(u\mid x)=Z_t(x)^{-1}\exp(-Q^{\pi_t}(x,u)/\tau)\),
\[
Q^{\pi_t}(x,u)+\tau\log\nu(u)
=
\tau\log\frac{\nu(u)}{p_t(u\mid x)}-\tau\log Z_t(x).
\]
The last term is constant in \(u\), and therefore disappears when paired with a signed perturbation of total mass zero. Hence, for any admissible perturbation \(\delta\nu\) with \(\int\delta\nu(u)\,\dd u=0\),
\[
D_\nu\mathcal J_t(\pi_t(\cdot\mid x);x)[\delta\nu]
=
\tau\int \log\frac{\pi_t(u\mid x)}{p_t(u\mid x)}\,\delta\nu(u)\,\dd u.
\]
Moreover, by~\eqref{eq:wpgf_re_form}, the WPGF can be written in relative-entropy gradient-flow form as
\[
\partial_t\pi_t
=
\tau\nabla_u\cdot\left(\pi_t\nabla_u\log\frac{\pi_t}{p_t}\right).
\]
Taking \(\delta\nu=\partial_t\pi_t(\cdot\mid x)\) and integrating by parts gives
\[
D_\nu\mathcal J_t(\pi_t(\cdot\mid x);x)[\partial_t\pi_t(\cdot\mid x)]
=
-\tau^2\int \left\|\nabla_u\log\frac{\pi_t(u\mid x)}{p_t(u\mid x)}\right\|^2\pi_t(\dd u\mid x).
\]
Equivalently,
\[
D_\nu\mathcal J_t(\pi_t(\cdot\mid x);x)[\partial_t\pi_t(\cdot\mid x)]
=
-\tau^2 \mathcal I(\pi_t(\cdot\mid x)\|p_t(\cdot\mid x)).
\]
Define
\[
g_t(x):=\tau^2 \mathcal I(\pi_t(\cdot\mid x)\|p_t(\cdot\mid x)).
\]
Combining the two contributions yields
\[
\dot V_t(x)=\gamma\Pop_t\dot V_t(x)-g_t(x),
\]
or equivalently
\[
(I-\gamma\Pop_t)\dot V_t=-g_t.
\]

By Lemma~\ref{lem:v2_resolvent}, the inverse \((I-\gamma\Pop_t)^{-1}\) is well defined on \(\Vtwo\) and admits the positive Neumann-series representation
\[
(I-\gamma\Pop_t)^{-1}=\sum_{k=0}^{\infty}\gamma^k\Pop_t^k.
\]
Therefore
\[
\dot V_t=-(I-\gamma\Pop_t)^{-1}g_t.
\]
Since \(g_t\ge0\) and the resolvent is positivity preserving, we obtain
\[
\dot V_t\le -g_t\le0.
\]
This proves~\eqref{eq:resolvent_identity} and~\eqref{eq:value_monotonicity}.
\end{proof}

\subsection{Objective dissipation along the linear-Gaussian parameter flow}

\begin{lemma}[Objective dissipation]\label{lem:objective_dissipation}
Along every admissible solution of the parameter ODE~\eqref{eq:KSigma_ode}, define
\[
G_t:=M_t-\frac{\tau}{2}\Sigma_t^{-1}.
\]
Then
\[
\dot\Sigma_t=-2(G_t\Sigma_t+\Sigma_tG_t),
\]
and
\[
\frac{\dd}{\dd t}C(K_t,\Sigma_t)
=
-4\Tr\!\left(\Scorr_{K_t,\Sigma_t}E_t^\top E_t\right)
-\frac{4}{1-\gamma}\Tr\!\left(\Sigma_tG_t^2\right)
\le 0.
\]
\end{lemma}

\begin{proof}
The covariance equation in~\eqref{eq:KSigma_ode} can be rewritten as
\[
\dot\Sigma_t
=
-2M_t\Sigma_t-2\Sigma_tM_t+2\tau I
=
-2(G_t\Sigma_t+\Sigma_tG_t).
\]
Using the gradient identities~\eqref{eq:gradients} and
\(\dot K_t=-2E_t\), we obtain
\begin{align*}
\frac{\dd}{\dd t}C(K_t,\Sigma_t)
&=
\left\langle 2E_t\Scorr_{K_t,\Sigma_t},-2E_t\right\rangle_F
+
\frac{1}{1-\gamma}
\left\langle G_t,-2(G_t\Sigma_t+\Sigma_tG_t)\right\rangle_F \\
&=
-4\Tr\!\left(\Scorr_{K_t,\Sigma_t}E_t^\top E_t\right)
-\frac{4}{1-\gamma}\Tr\!\left(\Sigma_tG_t^2\right).
\end{align*}
Both terms on the right-hand side are nonpositive because
\(\Scorr_{K_t,\Sigma_t}\succ0\), \(\Sigma_t\succ0\), and \(G_t\) is symmetric.
\end{proof}

\section{Stability, Compactness, and Global Well-Posedness}\label{app:stability}

This appendix proves the stability, compactness, and continuation results used for global well-posedness.

\subsection{Closed-loop detectability}

\begin{lemma}[Detectability passes to the closed loop]\label{lem:detectability_pass}
Under Assumption~\ref{ass:standard}, for every feedback matrix \(K\), the scaled pair
\[
(L_K^{1/2},\sqrt{\gamma}F_K)
\]
is detectable.
\end{lemma}

\begin{proof}
Suppose, to the contrary, that the pair
\[
(L_K^{1/2},\sqrt{\gamma}F_K)
\]
is not detectable. Then there exist \(v\ne0\) and \(\lambda\in\mathbb C\) such that
\[
F_Kv=\lambda v,
\qquad
|\sqrt{\gamma}\lambda|\ge1,
\qquad
L_K^{1/2}v=0.
\]
Since
\[
L_K=Q+K^\top RK\succeq0
\]
and \(R\succ0\), the identity \(L_K^{1/2}v=0\) implies
\[
Q^{1/2}v=0,
\qquad
Kv=0.
\]
Consequently,
\[
Av
=
F_Kv+BKv
=
\lambda v.
\]
Thus \(v\) is an unobservable eigenvector of the pair
\[
(Q^{1/2},\sqrt{\gamma}A)
\]
with \(|\sqrt{\gamma}\lambda|\ge1\), contradicting Assumption~\ref{ass:standard}.
Therefore \((L_K^{1/2},\sqrt{\gamma}F_K)\) is detectable.
\end{proof}

\subsection{Finite cost and discounted stability}

\begin{lemma}[Finite cost is equivalent to discounted stability]\label{lem:finite_adm}
Under Assumption~\ref{ass:standard}, for every \(K\) and every \(\Sigma\succ0\),
\[
C(K,\Sigma)<\infty
\qquad
\text{if and only if}
\qquad
\rho(\sqrt{\gamma}F_K)<1.
\]
\end{lemma}

\begin{proof}
First assume that
\[
\rho(\sqrt{\gamma}F_K)<1.
\]
Then the Lyapunov series defining \(P_K\) converges. Moreover, the discounted state
correlation
\[
\Scorr_{K,\Sigma}
=
\sum_{t=0}^{\infty}\gamma^t\E[x_tx_t^\top]
\]
is finite. Hence
\[
C(K,\Sigma)
=
\sum_{t=0}^{\infty}\gamma^t
\left\{
\Tr\!\left(L_K\E[x_tx_t^\top]\right)
+
\Tr(R\Sigma)
-
\frac{\tau}{2}
\left(
m+\log\bigl((2\pi)^m\det\Sigma\bigr)
\right)
\right\}
<\infty.
\]

Conversely, suppose that \(C(K,\Sigma)<\infty\). Define
\[
X_t:=\E[x_tx_t^\top].
\]
Since
\[
x_t=F_K^tx_0+\text{zero-mean noise independent of }x_0,
\]
we have
\[
X_t\succeq F_K^t\Gamma_0(F_K^t)^\top.
\]
For \(N\ge0\), define the finite-horizon observability Gramian
\[
P_N
:=
\sum_{t=0}^N
\gamma^t(F_K^t)^\top L_KF_K^t.
\]
Let
\[
c_\Sigma
:=
\Tr(R\Sigma)
-
\frac{\tau}{2}
\left(
m+\log\bigl((2\pi)^m\det\Sigma\bigr)
\right).
\]
Then
\[
C(K,\Sigma)
=
\sum_{t=0}^{\infty}\gamma^t\Tr(L_KX_t)
+
\frac{r_\Sigma}{1-\gamma}.
\]
Therefore, for every \(N\),
\[
\Tr(\Gamma_0P_N)
\le
\sum_{t=0}^N\gamma^t\Tr(L_KX_t)
\le
C(K,\Sigma)-\frac{r_\Sigma}{1-\gamma}
<\infty.
\]
Since \(\Gamma_0\succeq\mu I\), it follows that
\[
\mu\Tr(P_N)
\le
\Tr(\Gamma_0P_N)
\le
C(K,\Sigma)-\frac{c_\Sigma}{1-\gamma}.
\]
The sequence \((P_N)_{N\ge0}\) is monotone increasing in the positive semidefinite order
and has uniformly bounded trace. Hence it converges to some \(P\succeq0\). Passing to
the limit in
\[
P_N
=
L_K+\gamma F_K^\top P_{N-1}F_K
\]
gives
\[
P=L_K+\gamma F_K^\top PF_K.
\]

We now show that the closed loop must be discounted-stable. Suppose instead that
\[
\rho(\sqrt{\gamma}F_K)\ge1.
\]
Then there exist \(v\ne0\) and \(\lambda\in\mathbb C\) such that
\[
F_Kv=\lambda v,
\qquad
|\sqrt{\gamma}\lambda|\ge1.
\]
Testing the Lyapunov identity against \(v\) gives
\[
(1-\gamma|\lambda|^2)v^*Pv
=
v^*L_Kv.
\]
The left-hand side is nonpositive, while the right-hand side is nonnegative. Therefore
\[
v^*L_Kv=0,
\]
and hence
\[
L_K^{1/2}v=0.
\]
This contradicts Lemma~\ref{lem:detectability_pass}. Thus
\[
\rho(\sqrt{\gamma}F_K)<1.
\]
\end{proof}

\begin{remark}[The detectability hypothesis is necessary]
The detectability condition cannot be removed. For instance, if \(Q=0\), then
\[
(0,\sqrt{\gamma}A)
\]
is detectable only when \(\sqrt{\gamma}A\) is already stable. Consider the scalar
example
\[
A=2,
\qquad
B=1,
\qquad
Q=0,
\qquad
R=1,
\qquad
\gamma=\frac12,
\qquad
W=0,
\qquad
\Gamma_0=1,
\qquad
\tau=1.
\]
The unrestricted finite-cost minimizer is
\[
(K,\Sigma)=\left(0,\frac12\right),
\]
which is not discounted-stable. By contrast, the stationarity equations restricted to
the admissible class select a different admissible point, which is strictly suboptimal
for the unrestricted problem. Thus, without Assumption~\ref{ass:standard}(ii), finite
cost and discounted stability need not coincide.
\end{remark}

\subsection{Compactness of finite-cost sublevel sets}

\begin{proof}[Proof of Proposition~\ref{prop:stability_compactness}]
The finite-cost characterization is Lemma~\ref{lem:finite_adm}. It remains
to prove compactness of the sublevel sets.

\proofpart{Boundedness of sublevels} From the Lyapunov equation~\eqref{eq:PK_lyap},
\[
P_K
=
Q+K^\top RK+\gamma F_K^\top P_KF_K
\succeq K^\top RK.
\]
Hence, using~\eqref{eq:CKS}, \(\Gamma_0\succeq \mu I\), and \(R\succeq \lambda_R I\),
\[
\begin{aligned}
C(K,\Sigma)
&\ge
\Tr(\Gamma_0K^\top RK)
+
\frac{1}{1-\gamma}
\left(
\lambda_R\Tr(\Sigma)
-\frac{\tau}{2}\log\det\Sigma
\right)
-\frac{\tau m}{2(1-\gamma)}(1+\log(2\pi))  \\
&\ge
\mu\lambda_R\|K\|_F^2
+
\frac{1}{1-\gamma}
\left(
\lambda_R\Tr(\Sigma)
-\frac{\tau}{2}\log\det\Sigma
\right)
-\frac{\tau m}{2(1-\gamma)}(1+\log(2\pi)).
\end{aligned}
\]
The right-hand side tends to \(+\infty\) if
\(\|K\|_F\to\infty\), if \(\lambda_{\max}(\Sigma)\to\infty\), or if
\(\lambda_{\min}(\Sigma)\downarrow0\). Therefore every sublevel set
\[
\mathcal S_c:=\{(K,\Sigma)\in\cA_{\mathrm{adm}}: C(K,\Sigma)\le c\}
\]
is bounded, and its \(\Sigma\)-components remain in a compact subset of
\(\bbS_{++}^m\).

\proofpart{Closedness of sublevels}
Let \((K_j,\Sigma_j)\in\mathcal S_c\) be a convergent sequence with
\[
K_j\to K,
\qquad
\Sigma_j\to\Sigma\succ0.
\]
By the boundedness just proved, the matrices \(\Sigma_j\) remain in a compact subset of
\(\bbS_{++}^m\). Moreover, \(K_j\) is bounded. Hence the constants \(q_{K_j,\Sigma_j}\) in the value functions are uniformly bounded below: there exists a constant \(C_0\) such that
\[
q_{K_j,\Sigma_j}\ge C_0
\qquad \text{for all } j.
\]
Since \(C(K_j,\Sigma_j)\le c\), we get
\[
\Tr(\Gamma_0P_{K_j})
=
C(K_j,\Sigma_j)-q_{K_j,\Sigma_j}
\le c-C_0.
\]
Using \(\Gamma_0\succeq \mu I\), this implies
\[
\mu\Tr(P_{K_j})
\le
\Tr(\Gamma_0P_{K_j})
\le
c-C_0.
\]
Thus \((P_{K_j})\) is bounded in the cone of positive semidefinite matrices. Passing to a
subsequence if necessary, we may assume
\[
P_{K_j}\to P\succeq0.
\]
Taking limits in the Lyapunov equation
\[
P_{K_j}
=
Q+K_j^\top RK_j+\gamma F_{K_j}^\top P_{K_j}F_{K_j}
\]
gives
\[
P
=
Q+K^\top RK+\gamma F_K^\top PF_K.
\]
By Lemma~\ref{lem:detectability_pass} and the eigenvector argument used in
Lemma~\ref{lem:finite_adm}, this identity with \(P\succeq0\) implies
\[
\rho(\sqrt{\gamma}F_K)<1.
\]
Hence \((K,\Sigma)\) is admissible. On the admissible set, \(K\mapsto P_K\) is continuous,
and therefore \(C\) is continuous. Consequently,
\[
C(K,\Sigma)=\lim_{j\to\infty}C(K_j,\Sigma_j)\le c.
\]
Thus \(\mathcal S_c\) is closed. It is also bounded, its covariance matrices are uniformly positive definite, and every limit point satisfies the discounted-stability condition. Therefore \(\mathcal S_c\) is a compact subset of \(\cA_{\mathrm{adm}}\). In particular, \(C\) attains its minimum on each nonempty finite sublevel set.
By Assumption~\ref{ass:standard}(iii) and Lemma~\ref{lem:finite_adm}, such a
sublevel set is nonempty, and hence a minimizer exists.
\end{proof}

\subsection{Global well-posedness of the closed parameter flow}

\begin{proof}[Proof of Lemma~\ref{lem:global}]
By Assumption~\ref{ass:standard}(iii) and Lemma~\ref{lem:finite_adm}, the finite-cost initialization
\((K_0,\Sigma_0)\) is admissible. On the admissible set, the map
\(K\mapsto P_K\) is smooth. Indeed, the Lyapunov equation can be vectorized as
\[
\mathrm{vec}(P_K)
=
\left(
I-\gamma F_K^\top\otimes F_K^\top
\right)^{-1}
\mathrm{vec}(L_K),
\]
and the inverse exists whenever \(\rho(\sqrt{\gamma}F_K)<1\). Hence the vector
field in~\eqref{eq:KSigma_ode} is smooth on the admissible set, and the standard
ODE theorem gives a unique local solution.

Along this local solution, Lemma~\ref{lem:objective_dissipation} gives
\[
\frac{\dd}{\dd t}C(K_t,\Sigma_t)\le0.
\]
Therefore the trajectory remains in the finite sublevel set
\[
\mathcal S_{C(K_0,\Sigma_0)}
=
\{(K,\Sigma): C(K,\Sigma)\le C(K_0,\Sigma_0)\}.
\]
By Proposition~\ref{prop:stability_compactness}, this sublevel set is a compact subset of the admissible set. Since the trajectory remains in this compact set, it stays bounded, the covariance remains positive definite, and discounted stability is preserved. The standard continuation theorem for ODEs therefore extends the solution uniquely to all \(t\ge0\). Since the trajectory remains in the same
finite sublevel set, every point on the trajectory is admissible and has finite cost.
\end{proof}

\section{Proof of the Main Theorem}\label{app:main}

We use the Bellman resolvent identity and the LQ performance-difference identity to compare objective dissipation with the optimality gap.

\begin{proof}[Proof of Theorem~\ref{thm:main}]
By Lemma~\ref{lem:global}, the trajectory is globally defined, remains
admissible, and has finite cost for all \(t\ge0\).

For brevity, write
\[
C_t:=C(K_t,\Sigma_t),
\qquad
C_\star:=C(K_\star,\Sigma_\star),
\qquad
\Scorr_t:=\Scorr_{K_t,\Sigma_t},
\qquad
\Scorr_\star:=\Scorr_{K_\star,\Sigma_\star},
\]
and
\[
M_t:=M_{K_t},
\qquad
N_t:=N_{K_t},
\qquad
E_t:=E_{K_t},
\qquad
\Xi_t:=\frac{\tau}{2}M_t^{-1}.
\]
Let
\[
p_t(\cdot\mid x)
:=
\cN\!\left(-M_t^{-1}N_tx,\Xi_t\right)
\]
be the Gibbs policy associated with the current value function
\(V_t:=V_{K_t,\Sigma_t}\).

\proofpart{Step 1: Gaussian LSI and objective decrease}

Let \(\Pop_t:=\Pop_{\pi_t}\) denote the closed-loop transition operator
induced by \(\pi_t=\pi_{K_t,\Sigma_t}\), and define
\[
g_t(x)
:=
\tau^2
\mathcal I
\bigl(\pi_t(\cdot\mid x)\|p_t(\cdot\mid x)\bigr).
\]
By the Bellman resolvent identity~\eqref{eq:resolvent_identity},
\[
(I-\gamma\Pop_t)\dot V_t=-g_t.
\]
Therefore,
\[
-\frac{\dd}{\dd t}C_t
=
\E_{x_0\sim\cD}
\left[
\bigl((I-\gamma\Pop_t)^{-1}g_t\bigr)(x_0)
\right].
\]

The Gaussian log-Sobolev inequality~\eqref{eq:lq_lsi} gives
\[
\mathcal I(\nu\|p_t(\cdot\mid x))
\ge
2\alpha_t\KL\bigl(\nu\|p_t(\cdot\mid x)\bigr)
\ge
\frac{4\lambda_R}{\tau}
\KL\bigl(\nu\|p_t(\cdot\mid x)\bigr),
\qquad
\alpha_t
=
\frac{2\lambda_{\min}(M_t)}{\tau}.
\]
Taking \(\nu=\pi_t(\cdot\mid x)\) and using the Bellman residual
identity~\eqref{eq:bellman_kl},
\[
R_t(x)
:=
V_t(x)-(\cT_*V_t)(x)
=
\tau
\KL\bigl(\pi_t(\cdot\mid x)\|p_t(\cdot\mid x)\bigr),
\]
we obtain
\[
g_t(x)
\ge
4\lambda_R R_t(x).
\]
The temperature cancels in this inequality. This step introduces no factor of the form $\exp(-c/\tau)$.
By Lemma~\ref{lem:v2_resolvent}, the resolvent
\[
(I-\gamma\Pop_t)^{-1}
=
\sum_{j=0}^{\infty}\gamma^j\Pop_t^j
\]
is positivity preserving. It follows that
\begin{equation}
\label{eq:lq_lsi_dissipation}
-\frac{\dd}{\dd t}C_t
\ge
4\lambda_R\mathcal R_t,
\end{equation}
where
\[
\mathcal R_t
:=
\E_{x_0\sim\cD}
\left[
\bigl((I-\gamma\Pop_t)^{-1}R_t\bigr)(x_0)
\right].
\]

\proofpart{Step 2: LQ performance comparison}

The Gaussian KL formula~\eqref{eq:bellman_kl} for the Bellman residual is
\[
R_t(x)
=
x^\top E_t^\top M_t^{-1}E_tx+b_t,
\]
where
\[
b_t
:=
\frac{\tau}{2}
\left[
\Tr(\Xi_t^{-1}\Sigma_t)
-
m
-
\log\det(\Xi_t^{-1}\Sigma_t)
\right].
\]
Consequently, by the definition~\eqref{eq:state-correlation} of the discounted state-correlation matrix,
\[
\mathcal R_t
=
\Tr\!\left(
\Scorr_tE_t^\top M_t^{-1}E_t
\right)
+
\frac{b_t}{1-\gamma}.
\]

Recall that
\[
f_{K_t}(\Sigma)
=
\frac{\tau}{2(1-\gamma)}\log\det\Sigma
-
\frac{1}{1-\gamma}\Tr(M_t\Sigma).
\]
Since \(\Xi_t=\frac{\tau}{2}M_t^{-1}\), direct substitution gives
\[
f_{K_t}(\Xi_t)-f_{K_t}(\Sigma_t)
=
\frac{b_t}{1-\gamma}.
\]
Define
\[
B_t
:=
f_{K_t}(\Xi_t)-f_{K_t}(\Sigma_t).
\]
We thus have
\begin{equation}
\label{eq:Rt_expansion_lsi}
\mathcal R_t
=
\Tr\!\left(
\Scorr_tE_t^\top M_t^{-1}E_t
\right)
+
B_t.
\end{equation}

We next compare this expression with the global objective gap. Applying the
performance-difference identity~\eqref{eq:perf_diff} with
\[
(K,\Sigma)=(K_t,\Sigma_t),
\qquad
(K',\Sigma')=(K_\star,\Sigma_\star),
\]
and setting
\[
\Delta K_\star:=K_\star-K_t,
\]
yields
\[
C_t-C_\star
=
-\Tr\!\left(
\Scorr_\star
\left(
\Delta K_\star^\top M_t\Delta K_\star
+
2\Delta K_\star^\top E_t
\right)
\right)
+
f_{K_t}(\Sigma_\star)-f_{K_t}(\Sigma_t).
\]
The gain term satisfies
\[
\begin{aligned}
&\Tr\!\left(
\Scorr_\star
\left(
\Delta K_\star^\top M_t\Delta K_\star
+
2\Delta K_\star^\top E_t
\right)
\right) \\
&\quad =
\left\|
M_t^{1/2}\Delta K_\star\Scorr_\star^{1/2}
+
M_t^{-1/2}E_t\Scorr_\star^{1/2}
\right\|_F^2
-
\Tr\!\left(
\Scorr_\star E_t^\top M_t^{-1}E_t
\right) \\
&\quad \ge
-\Tr\!\left(
\Scorr_\star E_t^\top M_t^{-1}E_t
\right).
\end{aligned}
\]

Moreover, \(f_{K_t}\) is strictly concave on \(\bbS_{++}^m\), and
\[
\nabla_\Sigma f_{K_t}(\Sigma)
=
\frac{1}{1-\gamma}
\left(
\frac{\tau}{2}\Sigma^{-1}-M_t
\right).
\]
Its unique critical point is therefore
\[
\Sigma
=
\frac{\tau}{2}M_t^{-1}
=
\Xi_t,
\]
which is its unique global maximizer. Hence
\[
f_{K_t}(\Sigma_\star)
\le
f_{K_t}(\Xi_t),
\qquad
B_t\ge0.
\]
It follows that
\begin{equation}
\label{eq:gap_upper_lsi}
C_t-C_\star
\le
A_t+B_t,
\qquad
A_t
:=
\Tr\!\left(
\Scorr_\star E_t^\top M_t^{-1}E_t
\right).
\end{equation}

Let
\[
H_t:=E_t^\top M_t^{-1}E_t\succeq0.
\]
Since
\[
\Scorr_t\succeq\Gamma_0\succeq\mu I
\qquad\text{and}\qquad
\Scorr_\star\preceq\norm{\Scorr_\star}I,
\]
we have
\[
\Scorr_t
\succeq
\frac{\mu}{\norm{\Scorr_\star}}\Scorr_\star.
\]
Therefore,
\[
\Tr(\Scorr_tH_t)
\ge
\frac{\mu}{\norm{\Scorr_\star}}
\Tr(\Scorr_\star H_t)
=
\frac{\mu}{\norm{\Scorr_\star}}A_t.
\]
Also, \(\norm{\Scorr_\star}\ge\mu\), and hence, using \(B_t\ge0\),
\[
B_t
\ge
\frac{\mu}{\norm{\Scorr_\star}}B_t.
\]
Combining these inequalities with
\eqref{eq:Rt_expansion_lsi} and \eqref{eq:gap_upper_lsi} gives
\begin{equation}
\label{eq:Rt_gap_compare_lsi}
\mathcal R_t
\ge
\frac{\mu}{\norm{\Scorr_\star}}
(A_t+B_t)
\ge
\frac{\mu}{\norm{\Scorr_\star}}
(C_t-C_\star).
\end{equation}

\proofpart{Step 3: Exponential rate}

Combining~\eqref{eq:lq_lsi_dissipation} and
\eqref{eq:Rt_gap_compare_lsi}, we obtain
\[
-\frac{\dd}{\dd t}C_t
\ge
\frac{4\mu\lambda_R}{\norm{\Scorr_\star}}
(C_t-C_\star).
\]
Defining
\[
\lambda_\tau
:=
\frac{4\mu\lambda_R}{\norm{\Scorr_\star}},
\]
we have
\[
\frac{\dd}{\dd t}(C_t-C_\star)
\le
-\lambda_\tau(C_t-C_\star).
\]
This proves~\eqref{eq:pl_differential}. Since \(C_t\ge C_\star\) by
Proposition~\ref{prop:optimizer}, Gr\"onwall's inequality yields
\[
0
\le
C_t-C_\star
\le
e^{-\lambda_\tau t}(C_0-C_\star).
\]
This proves~\eqref{eq:main_rate}.
\end{proof}

\clearpage

\bibliographystyle{plainnat}
\bibliography{references}

@article{siska2026convergence,
  title={A Note on Convergence of Wasserstein Policy Optimization},
  author={David {\v{S}}i{\v{s}}ka and Yufei Zhang},
  journal={arXiv preprint arXiv:2605.22622},
  year={2026}
}

@article{zhu2026wasserstein,
  title={Wasserstein Proximal Policy Gradient},
  author={Zhu, Zhaoyu and Zhang, Shuhan and Gao, Rui and Li, Shuang},
  journal={arXiv preprint arXiv:2603.02576},
  year={2026},
  eprint={2603.02576},
  archivePrefix={arXiv},
  primaryClass={cs.LG}
}

@inproceedings{terpin2022ottrpo,
  title={Trust Region Policy Optimization with Optimal Transport Discrepancies: Duality and Algorithm for Continuous Actions},
  author={Terpin, Antonio and Lanzetti, Nicolas and Yardim, Batuhan and D{\"o}rfler, Florian and Ramponi, Giorgia},
  booktitle={Advances in Neural Information Processing Systems},
  volume={35},
  pages={19786--19797},
  year={2022},
  publisher={Curran Associates, Inc.}
}

@article{song2024metrictr,
  title={Provably Convergent Policy Optimization via Metric-aware Trust Region Methods},
  author={Song, Jun and He, Niao and Ding, Lijun and Zhao, Chaoyue},
  journal={Transactions on Machine Learning Research},
  year={2023},
}

@inproceedings{arbel2020kwng,
  title={Kernelized Wasserstein Natural Gradient},
  author={Arbel, Michael and Gretton, Arthur and Li, Wuchen and Montufar, Guido},
  booktitle={International Conference on Learning Representations},
  year={2020}
}

@inproceedings{moskovitz2020efficientwng,
  author       = {Ted Moskovitz and
                  Michael Arbel and
                  Ferenc Huszar and
                  Arthur Gretton},
  title        = {Efficient Wasserstein Natural Gradients for Reinforcement Learning},
  booktitle    = {9th International Conference on Learning Representations, {ICLR} 2021,
                  Virtual Event, Austria, May 3-7, 2021},
  publisher    = {OpenReview.net},
  year         = {2021}
}

@inproceedings{pfau2025wpo,
  title={{W}asserstein Policy Optimization},
  author={Pfau, David and Davies, Ian and Borsa, Diana L. and Ara{\'u}jo, Jo{\~a}o Guilherme Madeira and Tracey, Brendan Daniel and Van Hasselt, Hado},
  booktitle={Proceedings of the 42nd International Conference on Machine Learning},
  pages={49128--49149},
  year={2025},
  volume={267},
  series={Proceedings of Machine Learning Research},
  publisher={PMLR}
}

@inproceedings{zhang2018wgf,
  title={Policy optimization as wasserstein gradient flows},
  author={Zhang, Ruiyi and Chen, Changyou and Li, Chunyuan and Carin, Lawrence},
  booktitle={International Conference on Machine Learning},
  pages={5737--5746},
  year={2018},
  organization={PMLR}
}

@inproceedings{fazel2018lqrpg,
  title     = {Global Convergence of Policy Gradient Methods for the Linear Quadratic Regulator},
  author    = {Fazel, Maryam and Ge, Rong and Kakade, Sham and Mesbahi, Mehran},
  booktitle = {Proceedings of the 35th International Conference on Machine Learning (ICML)},
  series    = {Proceedings of Machine Learning Research},
  volume    = {80},
  pages     = {1467--1476},
  year      = {2018},
  publisher = {PMLR}
}

@book{bakry2014analysis,
  title={Analysis and Geometry of Markov Diffusion Operators},
  author={Bakry, Dominique and Gentil, Ivan and Ledoux, Michel},
  series={Grundlehren der mathematischen Wissenschaften},
  volume={348},
  year={2014},
  publisher={Springer}
}

@article{hambly2021lqr,
  title={Policy Gradient Methods for the Noisy Linear Quadratic Regulator over a Finite Horizon},
  author={Hambly, Ben and Xu, Renyuan and Yang, Huining},
  journal={SIAM Journal on Control and Optimization},
  volume={59},
  number={5},
  pages={3359--3391},
  year={2021}
}

@inproceedings{malik2019derivativefree,
  title={Derivative-Free Methods for Policy Optimization: Guarantees for Linear Quadratic Systems},
  author={Malik, Dhruv and Pananjady, Ashwin and Bhatia, Kush and Khamaru, Koulik and Bartlett, Peter L. and Wainwright, Martin J.},
  booktitle={Proceedings of the 22nd International Conference on Artificial Intelligence and Statistics (AISTATS)},
  series={Proceedings of Machine Learning Research},
  volume={89},
  pages={2916--2925},
  year={2019},
  publisher={PMLR}
}

@article{mohammadi2022lqr,
  title={Convergence and Sample Complexity of Gradient Methods for the Model-Free Linear--Quadratic Regulator Problem},
  author={Mohammadi, Hesameddin and Zare, Armin and Soltanolkotabi, Mahdi and Jovanovi{\'c}, Mihailo R.},
  journal={IEEE Transactions on Automatic Control},
  volume={67},
  number={5},
  pages={2435--2450},
  year={2022}
}

@book{anderson2007optimal,
  title={Optimal Control: Linear Quadratic Methods},
  author={Anderson, Brian D. O. and Moore, John B.},
  year={2007},
  publisher={Dover Publications}
}

@article{guo2026fast,
  title   = {Fast Policy Learning for Linear-Quadratic Control with Entropy Regularization},
  author  = {Guo, Xin and Li, Xinyu and Xu, Renyuan},
  journal = {SIAM Journal on Control and Optimization},
  volume  = {64},
  number  = {1},
  pages   = {124--151},
  year    = {2026}
}

@article{zhu2026globalwpg,
  title         = {Global Convergence of Wasserstein Policy Gradient for Entropy-Regularized Reinforcement Learning},
  author        = {Zhu, Zhaoyu and Gao, Rui and Li, Shuang},
  journal       = {arXiv preprint arXiv:2605.26078},
  year          = {2026},
  eprint        = {2605.26078},
  archivePrefix = {arXiv},
  primaryClass  = {cs.LG}
}

@article{giegrich2024convergence,
  title={Convergence of policy gradient methods for finite-horizon exploratory linear-quadratic control problems},
  author={Giegrich, Michael and Reisinger, Christoph and Zhang, Yufei},
  journal={SIAM Journal on Control and Optimization},
  volume={62},
  number={2},
  pages={1060--1092},
  year={2024},
  publisher={SIAM}
}

\end{document}